\documentclass[a4paper, reqno]{amsart}
\usepackage{amsaddr}

\usepackage[utf8]{inputenc}
\usepackage{hyperref}

\usepackage{amsthm, amsmath, amssymb}
\usepackage{xcolor}
\usepackage{enumitem}
\usepackage{mathtools}\mathtoolsset{showonlyrefs}

\usepackage{pgfplots}
\pgfplotsset{compat = newest}
\usepackage{caption}
\usepackage{makecell}
\usepackage{multirow}
\usepackage{tablefootnote}
\usepackage{tikz-cd}
\usepackage{tcolorbox}

\newtheorem{theorem}{Theorem}

\newtheorem{lemma}[theorem]{Lemma}
\newtheorem{proposition}[theorem]{Proposition}

\theoremstyle{definition}
\newtheorem{example}[theorem]{Example}
\newtheorem*{remark}{Remark}

\newcommand{\R}{\mathbb{R}}
\newcommand{\N}{\mathbb{N}}

\newcommand{\func}{f}

\newcommand{\funcmin}{\func_{\min}}

\newcommand{\infnorm}[2]{\| #2 \|_{#1}}

\newcommand{\positivecone}[1]{\mathcal{P}_+(#1)}
\newcommand{\quadmodule}[1]{{\mathcal{Q}}(#1)}
\newcommand{\preordering}[1]{\mathcal{T}(#1)}
\newcommand{\measures}{\mathcal{M}_+}

\newcommand{\bincube}[1]{\mathbb{B}^{#1}}
\newcommand{\ball}[1]{B^{#1}}
\newcommand{\simplex}[1]{\Delta^{#1}}
\newcommand{\mainset}{\mathbf{X}}

\newcommand{\lowbound}[1]{\mathrm{lb}(f, \quadmodule{\mainset})_{#1}}
\newcommand{\lowwbound}[1]{\mathrm{lb}(f, \preordering{\mainset})_{#1}}

\newcommand{\upbound}[1]{\mathrm{ub}(f, \Sigma[\x])_{#1}}
\newcommand{\uppbound}[1]{\mathrm{ub}(f, \quadmodule{\mainset})_{#1}}
\newcommand{\upppbound}[1]{\mathrm{ub}(f, \preordering{\mainset})_{#1}}

\newcommand{\x}{\mathbf{x}}
\newcommand{\y}{\mathbf{y}}
\renewcommand{\a}{\mathbf{a}}

\newcommand{\fvar}{\mathbf{u}}
\newcommand{\fvarr}{\mathbf{v}}

\newcommand{\kernel}{\mathrm{K}}
\newcommand{\kernelop}{\mathbf{K}}
\newcommand{\kernelCD}{\mathrm{C}}
\newcommand{\kernelCDn}[1]{\mathrm{C}^{(#1)}}
\newcommand{\kernelCDp}{\kernelCD}

\newcommand{\kernelopCD}{\mathbf{C}}
\newcommand{\kernelopCDn}[1]{\mathbf{C}^{(#1)}}

\newcommand{\poly}[1]{\mathcal{P}({#1})}

\newcommand{\gegen}[1]{\mathcal{G}^{(#1)}}
\newcommand{\gegensup}[1]{\widehat{\mathcal{G}}^{(#1)}}
\newcommand{\gegenorth}[1]{\widetilde{\mathcal{G}}^{(#1)}}
\newcommand{\gegenweight}[1]{w_{#1}}

\newcommand{\kraw}[1]{\mathcal{K}^{(#1)}}

\newcommand{\LS}[1]{{#1}}

\title[Sum-of-squares hierarchies and the Christoffel-Darboux kernel]{Sum-of-squares hierarchies for polynomial optimization and the Christoffel-Darboux kernel}
\author{Lucas Slot}
\address{Centrum Wiskunde en Informatica (CWI), Amsterdam.}
\email{lucas.slot@cwi.nl}
\thanks{This work is supported by the European Union's Framework Programme for Research and Innovation Horizon
2020 under the Marie Skłodowska-Curie Actions Grant Agreement No. 764759 (MINOA)}
\date{\today}

\begin{document}
\begin{abstract}
Consider the problem of minimizing a polynomial $f$ over a compact semialgebraic set ${\mathbf{X} \subseteq \R^n}$. Lasserre introduces hierarchies of semidefinite programs to approximate this hard optimization problem, based on classical sum-of-squares certificates of positivity of polynomials due to Putinar and Schm\"udgen. When $\mathbf{X}$ is the unit ball or the standard simplex, we show that the hierarchies based on the Schm\"udgen-type certificates converge to the global minimum of $f$ at a rate in $O(1/r^2)$, matching recently obtained convergence rates for the hypersphere and hypercube $[-1,1]^n$. For our proof, we establish a connection between Lasserre's hierarchies and the Christoffel-Darboux kernel, and make use of closed form expressions for this kernel derived by Xu.

\end{abstract}

\maketitle


\section{Introduction}
Let $\mainset \subseteq \R^n$ be a compact \emph{semialgebraic} set of the form:
\begin{equation}
\label{EQ:semialgebraic}
\mainset = \{ \x \in \R^n : g_j(\x) \geq 0 \quad (1 \leq j \leq m) \},
\end{equation}
where $g_j \in \R[\x]$ is a polynomial for each $j \in [m]$. We consider the problem of computing the global minimum:
\begin{equation}
    \label{EQ:mainproblem}
    \funcmin := \min_{\x \in \mainset} f(\x)
\end{equation}
of a polynomial $f$ of degree $d \in \N$ over the set $\mainset$. Polynomial optimization problems of the form \eqref{EQ:mainproblem} are very general. They capture hard optimization problems including \textsc{StableSet} and \textsc{MaxCut}, already for simple sets $\mainset$ such as the hypersphere, (binary) hypercube, unit ball and simplex. We refer to \cite{Laurent:polopt} and \cite{Lasserre:book} for an overview of further applications.

The program \eqref{EQ:mainproblem} may be reformulated as finding the largest $\lambda \in \R$ for which the polynomial $f - \lambda$ is nonnegative on $\mainset$. That is, writing $\positivecone{\mainset} \subseteq \R[\x]$ for the cone of all polynomials that are nonnegative on $\mainset$, we have:
\[
    \funcmin = \max \{ \lambda \in \R : f - \lambda \in \positivecone{\mainset} \}.
\]
This reformulation of \eqref{EQ:mainproblem} establishes a connection between polynomial optimization and the problem of certifying nonnegativity (or positivity) of a polynomial over a semialgebraic set. Using this connection and \emph{certificates of positivity} for polynomials on compact semialgebraic sets based on sums of squares, Lasserre \cite{Lasserre:lowbound, Lasserre:upbound} introduces several hierarchies of approximations of $\funcmin$ that may be computed efficiently using semidefinite programming.

\subsection{Positivity certificates and sum-of-squares hierarchies}
We say a polynomial $p \in \R[\x]$ is \emph{sum-of-squares} if it can be written as $p = p_1^2 + p_2^2 + \ldots + p_l^2$ for certain $p_i \in \R[\x]$. We write $\Sigma[\x] \subseteq \R[\x]$ for the set of all such polynomials. 
Consider the \emph{quadratic module} $\quadmodule{\mainset}$ and the \emph{preordering} $\preordering{\mainset}$ of $\mainset$, defined as:
\begin{alignat*}{2}
\quadmodule{\mainset} &:= \{\sum_{j = 0}^m \sigma_j g_j : \sigma_j \in \Sigma[\x]\} &&(\text{where } g_0 := 1),
\\
\preordering{\mainset} &:= \{\sum_{J \subseteq [m]} \sigma_J g_J : \sigma_J \in \Sigma[\x]\} \quad &&(\text{where } g_J := \prod_{j \in J} g_j).
\end{alignat*}
Note that strictly speaking, $\quadmodule{\mainset}$ and $\preordering{\mainset}$ do not depend on the set $\mainset$, but rather on its description \eqref{EQ:semialgebraic} as a semialgebraic set. We adopt this slight abuse of notation for clarity of exposition, as canonical descriptions are available for each of the sets $\mainset$ we consider.
As sum-of-squares polynomials are globally nonnegative, it is clear that:
\[
\Sigma[\x] \subseteq \quadmodule{\mainset} \subseteq \preordering{\mainset} \subseteq \positivecone{\mainset}.	
\]
Therefore, one may verify nonnegativity of a polynomial $f$ over $\mainset$ by showing that $f$ lies either in $\Sigma[\x], \quadmodule{\mainset}$ or $\preordering{\mainset}$.
\subsubsection{Hierarchies of lower bounds}
The key observation of Lasserre \cite{Lasserre:lowbound} is that membership of the \emph{truncated} quadratic module or preordering, defined as:
\begin{align*}
\quadmodule{\mainset}_{2r} &:= \{\sum_{j = 0}^m \sigma_j g_j : \sigma_j \in \Sigma[\x], ~\deg(\sigma_j g_j) \leq 2r \},
\\
\preordering{\mainset}_{2r} &:= \{\sum_{J \subseteq [m]} \sigma_J g_J : \sigma_J \in \Sigma[\x], ~\deg(\sigma_J g_J) \leq 2r\},
\end{align*}
may be checked by solving a semidefinite program whose size depends on $n, m$ and $r$ (see also \cite{deKlerkLaurent:survey}). This leads to the following hierarchies of \emph{lower} bounds on the global minimum $\funcmin$ of $f$ on~$\mainset$:
\begin{align}
	\label{EQ:lowbound}
	\lowbound{r} &:= \max \{ \lambda \in \R : f - \lambda \in \quadmodule{\mainset}_{2r} \}, 
	\\
	\label{EQ:lowwbound}
	\lowwbound{r} &:= \max \{ \lambda \in \R : f - \lambda \in \preordering{\mainset}_{2r} \}.
\end{align}
By definition, we have $\lowbound{r} \leq \lowwbound{r} \leq \funcmin$ for all $r \in \N$. Furthermore, the bounds converge to the global minimum $\funcmin$ as $r \to \infty$ under mild assumptions on $\mainset$. This is a consequence of the \emph{Positvstellens\"atze} of Putinar and Schm\"udgen, respectively.
\begin{theorem}[Putinar's Positvistellensatz] \label{THM:Putinar}
Let $\mainset \subseteq \R^n$ be a semialgebraic set, and assume that $R - \|\x\|^2 \in \quadmodule{\mainset}$ for some $R > 0$. Then for any polynomial $f \in \positivecone{\mainset}$ and $\eta > 0$, we have $f + \eta \in \quadmodule{\mainset}$.
\end{theorem}
\begin{theorem}[Schm\"udgen's Positvistellensatz] \label{THM:Schmudgen}
Let $\mainset \subseteq \R^n$ be a compact semialgebraic set. Then for any polynomial $f \in \positivecone{\mainset}$ and $\eta > 0$, we have $f + \eta \in \preordering{\mainset}$.
\end{theorem}
Semialgebraic sets $\mainset$ for which $R - \|\x\|^2$ lies in the quadratic module $\quadmodule{\mainset}$ are said to satisfy an \emph{Archimedean} condition. Note that such sets must be compact, and the requirement put on $\mainset$ in Theorem~\ref{THM:Putinar} is thus stronger than the one in Theorem~\ref{THM:Schmudgen}.

It should also be noted that polynomials positive on $\mainset$ need not be sum-of-squares (they need not even be \emph{globally} nonnegative). Therefore, a hierarchy of relaxations of the type \eqref{EQ:lowbound} where one instead demands that $f - \lambda \in \Sigma[\x]_{2r}$ does not converge to $\funcmin$ in general (the relaxation might not be feasible for \emph{any} $r \in \N$). This is true even when $\mainset = \R^n$ (for $n \geq 2$).

\subsubsection{Hierarchies of upper bounds.}
An alternative way to reformulate problem \eqref{EQ:mainproblem} is as follows:
\begin{equation} \label{EQ:measurereformulation}
	\funcmin = \inf_{\nu \in \measures(\mainset)} \bigg\{ \int_{\mainset} f(\x) d\nu(\x) : \int_{\mainset} d\nu(\x) = 1 \bigg\}.
\end{equation}
Here $\measures(\mainset)$ denotes the set of (positive) measures supported on $\mainset$. Indeed, we see that the optimum value of \eqref{EQ:measurereformulation} must be at least $\funcmin$, as we are taking the expectation of $f$ w.r.t. some probability measure on $\mainset$. On the other hand, choosing for $\nu$ the Dirac measure centered in a minimizer of $f$ over $\mainset$ shows that the optimum value of \eqref{EQ:measurereformulation} is at most $\funcmin$.

The idea of Lasserre \cite{Lasserre:upbound} now is to optimize not over the full set of measures on $\mainset$, but only over measures of the form $d\nu(\x) = q(\x) d\mu(\x)$, where $\mu$ is a \emph{fixed} reference measure supported on $\mainset$, and $q \in \R[\x]$ is a polynomial known to be nonnegative on $\mainset$. Such a relaxation yields an \emph{upper} bound on the global minimum $\funcmin$. Based on this observation, Lasserre \cite{Lasserre:upbound} defines for $r \in \N$:
\begin{align}
\label{EQ:upbound}
\upbound{r} &:= \inf_{q \in \Sigma[\x]_{2r}} \bigg\{ \int_{\mainset} f(\x) q(\x) d\mu(\x) : \int_{\mainset} q(\x) d\mu(\x) = 1 \bigg\}, \\
\label{EQ:uppbound}
\uppbound{r} &:= \inf_{q \in \quadmodule{\mainset}_{2r}} \bigg\{ \int_{\mainset} f(\x) q(\x) d\mu(\x) : \int_{\mainset} q(\x) d\mu(\x) = 1 \bigg\}, \\
\label{EQ:upppbound}
\upppbound{r} &:= \inf_{q \in \preordering{\mainset}_{2r}} \bigg\{ \int_{\mainset} f(\x) q(\x) d\mu(\x) : \int_{\mainset} q(\x) d\mu(\x) = 1 \bigg\}.
\end{align}
Each of these parameters can be computed by solving a semidefinite program whose size depends on $n, r$ (and the number of inequalities $m$ that define $\mainset$ in the case of $\uppbound{r}$ and $\upppbound{r}$, see also \cite{deKlerkLaurent:survey}). They satisfy:
\[
	\funcmin \leq \upppbound{r} \leq \uppbound{r} \leq \upbound{r} \quad (r \in \N).
\]
In contrast to the lower bounds, the upper bounds $\upbound{r}$ obtained by optimizing over $q \in \Sigma[\x]_{2r}$ already converge to the minimum $\funcmin$ of $f$ on $\mainset$ as $r \to \infty$ (under mild conditions on $\mainset$, $\mu$)~\cite{Lasserre:upbound}.
Note that the value of the upper bounds depends on the choice of reference measure $\mu$, which is supressed in the notation.

\subsection{Related work}
Recently, there has been an active interest in understanding the (asymptotic) behaviour of the Lasserre hierarchies defined above. We outline here some of the main known results. See also Table~\ref{TAB:overview:lowbounds} and Table~\ref{TAB:overview:upbounds} below.

\subsubsection{Lower bounds.} 
\sloppy
For general Archimedean semialgebraic sets $\mainset$, it is known~\cite{NieSchweighofer:putinar} that the Putinar-type bounds $\lowbound{r}$ converge to $\funcmin$ at a rate in $O(1 / \log (r)^c)$, where $c > 0$ is a constant depending on $\mainset$. When $\mainset = S^{n-1}$ is the hypersphere, an improved convergence rate in $O(1/r^2)$ may be shown \cite{FangFawzi:sphere}. For the binary hypercube $\mainset = \{0, 1\}^n$, specialized results are also available. For instance, it is known that the lower bound $\lowbound{r}$ is exact when $r \geq (n+d-1)/2$ \cite{FawziSaundersonParrilo:bincube, Sakaueetal:bincube}. For $r = \Omega(n)$, the convergence rate of $\lowbound{r}$ on $\mainset = \{0, 1\}^n$ may be expressed in terms of the roots of certain classical orthogonal polynomials \cite{SlotLaurent:bincube}. We also want to mention the recent work~\cite{Magron:PV}, where a rate in $O(1/r^c)$ is shown for a Lasserre-type hierarchy of lower bounds based on Putinar-Vasilescu's Positivstellensatz, which is outside the scope of this paper.

\sloppy
For general compact semialgebraic sets $\mainset$, the Schm\"udgen-type bounds $\lowwbound{r}$ converge to $\funcmin$ at a rate in $O(1 / r^c)$, where $c > 0$ again is a constant depending on $\mainset$ \cite{Schweighofer:schmudgen}.
In the special case that $\mainset = [-1,1]^n$ is the hypercube, a convergence rate in $O(1 / r)$ was shown in \cite{deKlerkLaurent:hypercube2010}, and recently improved to $O(1/r^2)$ in \cite{LaurentSlot:hypercube}. When $\mainset = \simplex{n}$, a rate in $O(1/r)$ is known \cite{deKlerkKirschner:simplex}.

\LS{Finally, in the very  recent work~\cite{BaldiMourrain:putinar}, the authors show a convergence rate in $O(1 / r^c)$ for the \emph{Putinar}-type bounds on general Archimedean semialgebraic sets. They thus match the best known (general) rate for the Schm\"udgen-type bounds and improve exponentially on the previous best known rate of \cite{NieSchweighofer:putinar}.}

\subsubsection{Upper bounds.} Turning now to the upper bounds, we have convergence of $\upbound{r}$ to $\funcmin$ at a rate in $O(\log^2 r / r^2)$ for general, compact, full-dimensional semialgebraic sets $\mainset$ equipped with the Lebesgue measure~\cite{SlotLaurent:upbound2}. When $\mainset$ is the hypersphere $S^{n-1}$, the hypercube $[-1,1]^n$, the standard simplex $\simplex{n}$ or the unit ball $B^n$ a rate in $O(1/r^2)$ may be shown \cite{deKlerkLaurent:sphere, deKlerkLaurent:hypercube2020, SlotLaurent:upbound} for natural choices of the reference measure $\mu$. This rate is best-possible in a certain sense \cite{deKlerkLaurent:hypercube2020}. Convergence rates for the stronger bounds $\uppbound{r}$ and $\upppbound{r}$ follow immediately. No examples are known of seperation between the asymptotic behaviour of the upper bounds based on optimization over $\Sigma[\x], \quadmodule{\mainset}$ and $\preordering{\mainset}$, respectively.

\subsection{Our contributions}
The main contribution of this paper is to show a convergence rate in $O(1/r^2)$ of the lower bounds $\lowwbound{r}$ to the global minimum $\funcmin$ of a polynomial $f$ on the unit ball 
$B^n$ or on the standard simplex $\simplex{n}$. These results can also be interpreted as showing a bound in $O(1 / \sqrt{\eta})$ on the degree of a Schm\"udgen-type certificate of positivity of $f + \eta$ for $f$ nonnegative over $B^n$ or $\simplex{n}$, respectively, meaning that $c \geq 2$ in the general result \cite{Schweighofer:schmudgen} for these sets.

\begin{theorem}\label{THM:mainball}
Let $\mainset = \ball{n} = \{\x \in \R^n : \|x\|^2 \leq 1 \}$ be the $n$-dimensional unit ball and let $f \in \R[\x]$ be a polynomial of degree $d$. Then for any $r \geq 2nd\sqrt{d}$, the lower bound $\lowwbound{r}$ for the minimization of $f$ over $\ball{n}$ satisfies:
\[
	\funcmin - \lowwbound{r} \leq \frac{C_B(n,d)}{r^2} \cdot (f_{\max} - \funcmin).
\]
Here, $C_B(n,d)$ is a constant depending only on $n, d$. This constant depends polynomially on $n$ (for fixed $d$) and polynomially on $d$ (for fixed $n$). 
See relation \eqref{EQ:ballconstant} for details.
\end{theorem}

\begin{theorem}\label{THM:mainsimplex}
Let $\mainset = \simplex{n} = \{ \x \in \R^n : 1 - \sum_i x_i \geq 0, ~x_i \geq 0 \quad (1 \leq i \leq n)\}$ be the $n$-dimensional standard simplex and let $f \in \R[\x]$ be a polynomial of degree $d$. Then for any $r \geq 2nd\sqrt{d}$, the lower bound $\lowwbound{r}$ for the minimization of $f$ over $\simplex{n}$ satisfies:
\[
	\funcmin - \lowwbound{r} \leq \frac{C_{\simplex{}}(n,d)}{r^2} \cdot (f_{\max} - \funcmin).
\]
Here, $C_{\simplex{}}(n,d)$ is a constant depending only on $n, d$. This constant depends polynomially on $n$ (for fixed $d$) and polynomially on $d$ (for fixed $n$). See relation \eqref{EQ:simplexconstant} for details.
\end{theorem}

\LS{Our second contribution is primarily expository; namely we unify and reinterpret some of the earlier techniques used to analyze the behaviour of the Lasserre-type bounds defined above.} We show a connection between the behaviour of the \emph{lower} bounds defined in \eqref{EQ:lowbound}, \eqref{EQ:lowwbound}, the \emph{upper} bounds defined in \eqref{EQ:upbound}, \eqref{EQ:uppbound}, \eqref{EQ:upppbound}, and the celebrated \emph{Christoffel-Darboux kernel}. As we explain in more detail below, this connection relies on an application of the \emph{polynomial kernel method}, employing a perturbed version of the Christoffel-Darboux kernel. An analysis is then possible in special cases (which include the hypersphere~\cite{FangFawzi:sphere}, binary hypercube~\cite{SlotLaurent:bincube}, unit ball and simplex) where this kernel admits a \emph{closed form}. 

\LS{One motivation for making the connection to the Christoffel-Darboux kernel explicit is that this kernel has recently seen increased interest in the context of (polynomial) optimization \cite{DeCastoetal:CD,  LasserrePauwels:empericalCD, MarxetAl:CD, PauwelsPutinarLasserre:CD}.} Of particular \LS{relevance} is the recent work~\cite{Lasserre:CDkernel}, where a link is established between this kernel and the hierarchy of lower bounds~\eqref{EQ:lowbound} \LS{(although this link is entirely different from the one we present below)}. 

As our final contribution, we obtain a convergence rate in $O(1/r^2)$ for the \emph{upper} bounds $\upppbound{r}$ on the unit ball and simplex w.r.t. to reference measures $\mu_{\ball{}}$ and $\mu_{\simplex{}}$ defined below in \eqref{EQ:ballmeasure} and \eqref{EQ:simplexmeasure}. \LS{The aforementioned connection between the Christoffel-Darboux kernel, the lower bounds and the upper bounds actually makes it quite elementary to prove these bounds (see Section~\ref{SEC:upbounds}).} However, the obtained rates do not improve upon previous results. Indeed, it is known \cite{SlotLaurent:upbound} that even the weaker bounds $\upbound{r}$ already converge to the global minimum at a rate in $O(1/r^2)$ for these sets (although for different reference measures).

\begin{theorem}\label{THM:secondball}
Let $\mainset = \ball{n}$ be the $n$-dimensional unit ball equipped with the measure $\mu_{\ball{}}$ defined in \eqref{EQ:ballmeasure}. Let $f \in \R[\x]$ be a polynomial of degree $d$. Then for any $r \geq d$, the upper bound $\upppbound{r}$ for the minimization of $f$ over $\ball{n}$ satisfies:
\[
	\upppbound{r} - \funcmin \leq \frac{C_B(n,d)}{2r^2} \cdot (f_{\max} - \funcmin).
\]
Here, $C_B(n,d)$ is the constant of Theorem~\ref{THM:mainball}.
\end{theorem}

\begin{theorem}\label{THM:secondsimplex}
Let $\mainset = \simplex{n}$ be the $n$-dimensional standard simplex equipped with the measure $\mu_{\simplex{}}$ defined in \eqref{EQ:simplexmeasure}. Let $f \in \R[\x]$ be a polynomial of degree $d$. Then for any $r \geq d$, the Lasserre-type upper bound $\upppbound{r}$ for the minimization of $f$ over $\simplex{n}$ satisfies:
\[
	\upppbound{r} - \funcmin \leq \frac{C_{\simplex{}}(n,d)}{2r^2} \cdot (f_{\max} - \funcmin).
\]
Here, $C_{\simplex{}}(n,d)$ is the constant of Theorem~\ref{THM:mainsimplex}.
\end{theorem}

\begin{table}[ht!]
{\renewcommand{\arraystretch}{1.2}%
\begin{small}
\begin{tabular}{cclc}
$\mainset$ \textbf{(compact)} & \textbf{error} & \textbf{certificate} & \textbf{reference}\\
\hline
Archimedean   & $O(1/\log(r)^c)$ & $\quadmodule{\mainset}$ & \cite{NieSchweighofer:putinar} \\
\LS{Archimedean}   & \LS{$O(1/r^c)$} & \LS{$\quadmodule{\mainset}$} & \LS{\cite{BaldiMourrain:putinar}} \\
general  	  & $O(1/r^c)$ & $\preordering{\mainset}$ & \cite{Schweighofer:schmudgen} \\
$S^{n-1}$ 	  & $O(1/r^2)$ & $\quadmodule{\mainset} ~(= \preordering{\mainset})$ %
& \cite{FangFawzi:sphere} \\
$\{0, 1\}^n$ & see \cite{SlotLaurent:bincube} & $\quadmodule{\mainset} ~(= \preordering{\mainset})$ & \cite{SlotLaurent:bincube} \\
$\ball{n}$ 	  & $O(1/r^2)$ & $\quadmodule{\mainset} ~(=\preordering{\mainset})$ & Theorem~\ref{THM:mainball} \\
$[-1, 1]^n$   & $O(1/r^2)$ & $\preordering{\mainset}$ & \cite{LaurentSlot:hypercube} \\
$\simplex{n}$ & $O(1/r)$ & $\preordering{\mainset}$ & \cite{deKlerkKirschner:simplex} \\
$\simplex{n}$ & $O(1/r^2)$ & $\preordering{\mainset}$ & Theorem~\ref{THM:mainsimplex} \\
\hline
\end{tabular}
\end{small}
}
\caption{\small Overview of known and new results on the asymptotic error of Lasserre's hierarchies of lower bounds.}
\label{TAB:overview:lowbounds}
\end{table}

\begin{table}[h!] 
{\renewcommand{\arraystretch}{1.2}%
\begin{small}
\begin{tabular}{ccccc}
$\mainset$ \textbf{(compact)} & \textbf{error} & \textbf{certificate} & \textbf{measure} & \textbf{reference}\\
\hline 
\makecell{full-dimensional} & $O(\log^2(r)/r^2)$ & $\Sigma[\x]$ & Lebesgue &\cite{SlotLaurent:upbound2} \\
$S^{n-1}$ 	  & $O(1/r^2)$ & $\Sigma[\x]$ & uniform (Haar) & \cite{deKlerkLaurent:sphere} \\
$[-1, 1]^n$   & $O(1/r^2)$ & $\Sigma[\x]$ & \makecell{$\prod_{i}(1-x_i)^\lambda d\x~(\lambda \geq -\frac{1}{2})$} & \cite{deKlerkLaurent:hypercube2020, SlotLaurent:upbound} \\
$\{0, 1\}^n$ & see \cite{SlotLaurent:bincube} & $\Sigma[\x]$ & uniform &\cite{SlotLaurent:bincube} \\
$\ball{n}$ 	  & $O(1/r^2)$ & $\Sigma[\x]$ & $(1-\|\x\|^2)^\lambda d\x~(\lambda \geq 0)$ & \cite{SlotLaurent:upbound} \\
$\simplex{n}$ & $O(1/r^2)$ & $\Sigma[\x]$ & Lebesgue & \cite{SlotLaurent:upbound} \\
$[-1, 1]^n$   & {$O(1/r^2)$} & $\preordering{\mainset}$ & $\prod_{i}(1-x_i)^{-\frac{1}{2}} d\x$ & \cite{deKlerkHessLaurent:hypercube} \\
$\ball{n}$    & {$O(1/r^2)$} & $\preordering{\mainset}$ & $\mu_{\ball{}}$, see \eqref{EQ:ballmeasure} & Theorem~\ref{THM:secondball} \\
$\simplex{n}$ & {$O(1/r^2)$} & $\preordering{\mainset}$ & $\mu_{\simplex{}}$, see \eqref{EQ:simplexmeasure} & Theorem~\ref{THM:secondsimplex} \\
\hline
\end{tabular}
\end{small}
}
\caption{\small Overview of known and new results on the asymptotic error of Lasserre's hierarchies of upper bounds.}
\label{TAB:overview:upbounds}
\end{table}

\subsection{Outline of the proof technique for the main results} \label{SEC:outline}
We give an overview of the technique we use to prove our main results Theorem~\ref{THM:mainball} and Theorem~\ref{THM:mainsimplex}. A similar technique was used to prove convergence results for other semialgebraic sets; on the hypersphere \cite{FangFawzi:sphere}, on the binary hypercube $\{0, 1\}^n$ \cite{SlotLaurent:bincube} and on the regular hypercube $[-1, 1]^n$ in \cite{deKlerkLaurent:hypercube2010} and \cite{LaurentSlot:hypercube}.

Let $f \in \R[\x]$ be a polynomial of degree $d \in \N$. We wish to show that 
\[
f - \funcmin + \epsilon \in \preordering{\mainset}_{2r}
\]
for some small $\epsilon > 0$. Up to translation and scaling, we may 
assume that $\funcmin = 0$ and that $\infnorm{\mainset}{f} := \max_{\x \in \mainset} |f(x)| = 1$.
Let $\epsilon > 0$. Suppose that we are able to construct an (invertible) linear operator ${\kernelop:\R[\x]_d\to\R[\x]_d}$ which satisfies the following three properties:
\begin{align}
	\label{PROPERTY:normalization}
	&\kernelop (1) = 1, \tag{P1}\\ 
	\label{PROPERTY:incone}
	&\kernelop p \in \preordering{\mainset}_{2r} \quad \text{ for all } p \in \positivecone{\mainset}_d \tag{P2} \\
	\label{PROPERTY:infnorm}
    &\max_{\x \in \mainset} |\kernelop^{-1} f(\x) - f(\x)| \leq \epsilon \tag{P3}.
\end{align}
We claim that we then have $f + \epsilon \in \preordering{\mainset}_{2r}$. Indeed, since $f$ is nonnegative on $\mainset$ by assumption, we know that $f(\x)+\epsilon \geq \epsilon$ for $\x \in \mainset$. By properties \eqref{PROPERTY:normalization} and \eqref{PROPERTY:infnorm}, it follows that $\kernelop^{-1}(f+\epsilon) \in \positivecone{\mainset}$. Using property \eqref{PROPERTY:incone}, we may thus conclude that:
\[
	f + \epsilon = \kernelop \big( \kernelop^{-1}(f + \epsilon) \big) \in \preordering{\mainset}_{2r},
\]
meaning that $\funcmin - \lowwbound{r} \leq \epsilon$.
The statements of Theorem~\ref{THM:mainball} and Theorem~\ref{THM:mainsimplex} may thus be proven by showing the existence (for each $r \in \N$ large enough) of an operator $\kernelop$ which satisfies \eqref{PROPERTY:normalization}, \eqref{PROPERTY:incone} and \eqref{PROPERTY:infnorm} with $\epsilon = O(1/r^2)$. We summarize this observation in the following Lemma for future reference.
\begin{lemma} \label{LEM:overviewsummary}
Let $\mainset \subseteq \R^n$ be a compact semialgebraic set and let $f$ be a polynomial on $\mainset$ of degree $d$. Suppose that there exists a nonsingular linear operator ${\kernelop : \R[\x]_d \to \R[\x]_d}$ which satisfies the properties \eqref{PROPERTY:normalization}, \eqref{PROPERTY:incone} and \eqref{PROPERTY:infnorm} for certain $\epsilon \geq 0$. Then $\funcmin - \lowwbound{r} \leq \epsilon$.
\end{lemma}

One way of constructing operators that satisfy \eqref{PROPERTY:normalization}, \eqref{PROPERTY:incone} and \eqref{PROPERTY:infnorm} is by applying the \emph{kernel polynomial method}: Let $\kernel : \mainset \times \mainset \to \R$ be a polynomial kernel on $\mainset$, \LS{meaning that $\kernel(\x, \y)$ is a polynomial in the variables $\x, \y$}. After choosing a measure $\mu$ supported on $\mainset$, we may associate a linear operator $\kernelop : \R[\x] \to \R[\x]$ to $\kernel$ by setting:
\begin{equation} \label{EQ:associatedkernel}
	\kernelop p (\x) := \int_{\mainset} \kernel(\x, \y) p(\y) d\mu(\y) \quad (p \in \R[\x]).
\end{equation}
It turns out that since $\preordering{\mainset}_{2r} \subseteq \R[\x]$ is a convex cone, the operator $\kernelop$ satisfies \eqref{PROPERTY:incone} if the polynomial $\x \mapsto \kernel(\x, \y)$ lies in $\preordering{\mainset}_{2r}$ for all \emph{fixed} $\y \in \mainset$ (see Lemma~\ref{LEM:incone}). Furthermore, it turns out that \eqref{PROPERTY:normalization}, \eqref{PROPERTY:infnorm} may be verified by analyzing the \emph{eigenvalues} of $\kernelop$; roughly speaking, $\kernelop$ satisfies \eqref{PROPERTY:normalization}, \eqref{PROPERTY:infnorm} if its eigenvalues are sufficiently close to $1$ (see Section~\ref{SEC:P2}).

It remains, then, to construct a suitable kernel $\kernel$ on $\mainset$. The starting point of our construction is the \emph{Christoffel-Darboux} kernel $\kernelCD_{2r}$ of degree $2r$, which is defined in terms of an orthornormal basis $\{ P_\alpha : \alpha \in \N^n \}$ for $\R[\x]$ w.r.t. $(\mainset, \mu)$ as:
\[
\kernelCD_{2r}(\x, \y) := \sum_{k = 0}^{2r} \kernelCDn{k}(\x, \y), \quad \text{where } \kernelCDn{k}(\x, \y) := \sum_{|\alpha| = k} P_\alpha(\x) P_\alpha(\y).
\]
The operator $\kernelopCD_{2r}$ associated to $\kernelCD_{2r}$ via \eqref{EQ:associatedkernel} \emph{reproduces} the space of polynomials of degree up to $2r$, i.e., it satisfies:
\[
	\kernelopCD_{2r} p(\x) = p(\x) \quad(\x \in \mainset, \quad p \in \R[\x]_{2r}).
\]
In other words, its eigenvalues are all equal to $1$. The idea now is to perturb the Christoffel-Darboux kernel, considering instead the kernel:
\begin{equation} \label{EQ:overvieweq1}
	\kernelCDp_{2r}(\x, \y; \lambda) := \sum_{k = 0}^{2r} \lambda_{k} \kernelCDn{k}(\x, \y) \quad (\lambda_k \in \R, \quad 0 \leq k \leq 2r),
\end{equation}
whose associated operator has eigenvalues equal to $\lambda_0, \ldots, \lambda_{2r}$ (with multiplicity).

As we show below in Section \ref{SEC:CD}, conditions may be derived on the $\lambda_k$ to ensure that $\kernelCDp_{2r}(\cdot, \y; \lambda) \in \preordering{\mainset}_{2r}$ for certain special sets $\mainset$, inluding the hypersphere, the binary hypercube, the unit ball and the standard simplex (equipped with suitable measures). These conditions are based on a \emph{closed form} of the Christoffel-Darboux kernel, and they are met when the $\lambda_k$ are related to the coefficients of a \emph{univariate} sum-of-squares polynomial $q \in \Sigma[x]_{2r}$ in an appropriate basis of \emph{orthogonal polynomials}.
The problem of finding an operator satisfying \eqref{PROPERTY:normalization}, \eqref{PROPERTY:incone} and \eqref{PROPERTY:infnorm} then reduces to finding a sum of squares $q \in \Sigma[x]_{2r}$ for which these coefficients satisfy $\lambda_0 = 1$ and $\lambda_k \approx 1$ for $1 \leq k \leq d$.

Finally, as we discuss in Section~\ref{SEC:P2}, the problem of finding such $q$ is related closely to the hierarchies of \emph{upper} bounds \eqref{EQ:upbound}. These hierarchies are well-understood in the univariate case, and their behaviour may be expressed in terms of the \emph{roots} of the appropriate orthogonal polynomials. 

\begin{example} \label{EX:sphere-cube}
Consider the cases $\mainset = S^{n-1}$ and $\mainset = \{0, 1\}^n$, which were treated in \cite{FangFawzi:sphere} and \cite{SlotLaurent:bincube}, respectively. These sets are very well-structured, which allows the Christoffel-Darboux kernel to be written in a particularly simple form.

Let $q \in \R[x]_{2r}$ be a univariate polynomial. In the case that $\mainset = S^{n-1}$ is the hypersphere (equipped with the uniform Haar measure), the \emph{Funk-Hecke formula} states that:
\begin{equation} \label{EQ:overviewex1}
	q ( \x \cdot \y) = \kernelCDp_{2r}(\x, \y; \lambda) = \sum_{k=0}^{2r} \lambda_k \kernelCDn{k}(\x, \y) \quad (\x, \y \in S^{n-1}),
\end{equation}
where $\lambda = (\lambda_k)$ is given by the coefficients of $q$ in the basis of \emph{Gegenbauer} polynomials (see Section~\ref{SEC:gegenbauer}).

If $q$ is a sum of squares, the polynomial $\x \mapsto q ( \x \cdot \y)$ lies in $\preordering{S^{n-1}}_{2r}$, so that the operator associated to the above kernel satisfies \eqref{PROPERTY:incone} (see Lemma~\ref{LEM:incone}). It remains to show that it is possible to choose $q$ in such a way that \eqref{PROPERTY:infnorm} is satisfied with $\epsilon = O(1/r^2)$, which may be done using results on the roots of the Gegenbauer polynomials (see Section~\ref{SEC:P2}).

When $\mainset = \{0, 1\}^n$ is the binary cube (equipped with the uniform measure), one similarly has%
\footnote{Note that the polynomial $\sum_{i=1}^n x_i - 2x_iy_i + y_i$ coincides with $|\x-\y|$ for ${\x,\y \in \{0,1\}^n}$.}%
:
\begin{equation} \label{EQ:overviewex2}
	q (|\x - \y|) = \kernelCDp_{2r}(\x, \y; \lambda) \quad (\x, \y \in \{0,1\}^n).
\end{equation}
The associated operator again satisfies \eqref{PROPERTY:incone}, and its eigenvalues $\lambda_k$ are now given by the expansion of $q$ into the basis of \emph{Krawtchouk} polynomials. This allows one to show \eqref{PROPERTY:infnorm}, now with a bound on $\epsilon$ in terms of the roots of the Krawtchouk polynomials.

As we see below in Section \ref{SEC:CD}, the formulas \eqref{EQ:overviewex1} and \eqref{EQ:overviewex2} may both be seen as a consequence of a closed form of the Chrisoffel-Darboux kernel.
\end{example}

\subsection*{Organization}
\sloppy
The rest of the paper is organized as follows. In Section~\ref{SEC:Prelim}, we introduce some notations and cover the necessary preliminaries on orthogonal (Gegenbauer) polynomials. In Section~\ref{SEC:CD}, we present closed form expressions of the Christoffel-Darboux kernel and use them to obtain kernels whose associated operators satisfy \eqref{PROPERTY:incone} and whose eigenvalues are given by the coefficients of a univariate sum of squares in an appropriate basis of orthogonal polynomials. In Section~\ref{SEC:P2}, we show how to choose this sum of squares so that \eqref{PROPERTY:normalization}, \eqref{PROPERTY:infnorm} are satisfied and finish the proof of Theorem~\ref{THM:mainball} and Theorem~\ref{THM:mainsimplex}. In Section~\ref{SEC:upbounds}, we extend our proof technique to the upper bounds to obtain Theorem~\ref{THM:secondball} and Theorem~\ref{THM:secondsimplex}. Finally, we give a proof of some technical statements in Appendix~\ref{APP:sosrepresentations}.

\section{Preliminaries} \label{SEC:Prelim}
\subsection{Notations}
We write $\R[x]$ for the univariate polynomial ring, while reserving the bold-face notation $\R[\x] = \R[x_1, x_2, \dots, x_n]$ to denote the ring of polynomials in $n$ variables. 
Similarly, $\Sigma[x] \subseteq \R[x]$ and $\Sigma[\x] \subseteq \R[\x]$ denote the sets of univariate and $n$-variate sum-of-squares polynomials, respectively, consisting of all polynomials of the form $p = p_1^2 + p_2^2 + \dots + p_l^2$ for certain polynomials $p_1,\ldots,p_l$ and $l \in \N$. 
For a semialgebraic set $\mainset \subseteq \R^n$, we write $\poly{\mainset}$ for the space of polynomials on $\mainset$, which is the quotient $\R[\x] / \sim_{\mainset}$, where $p \sim_{\mainset} q$ if and only if $p(\x) = q(\x)$ for all $\x \in \mainset$. For such a polynomial, we denote $\infnorm{\mainset}{p} := \max_{\x \in \mainset} |p(\x)|$ for its supremum-norm on $\mainset$. We call a univariate polynomial $p$ \emph{even} if $p(x) = p(-x)$ for all $x \in \R$, and \emph{odd} if $p(x) = -p(-x)$ for all $x \in \R$. Finally, we write $|\x| := \sum_{i=1}^n x_i$ for $\x \in \R^n$.

\subsection{Quadratic modules and preorderings} \label{SEC:modules}
We are primarily concerned with the semialgebraic sets:
\begin{align*}
	S^{n-1} &= \{ \x \in \R^n : \|\x\|^2 = 1 \}, \\
	\{0, 1\}^n &= \{ \x \in \R^n : x_i^2 = x_i \quad (1 \leq i \leq n)\}, \\
	\ball{n} &= \{ \x \in \R^n : \|\x\|^2 \leq 1 \}, \\
	\simplex{n} &= \{ \x \in \R^n : 1 - |\x| \geq 0, ~ x_i \geq 0 \quad (1 \leq i \leq n)\}.
\end{align*}
We make note of a few special properties of the quadratic modules and preorderings associated to them.
Since $S^{n-1}$ and $\{0,1\}^n$ are defined only in terms of equalities, we have:
\begin{equation} \label{EQ:spherebincubemodule}
	\quadmodule{\mainset}_{2r} = \preordering{\mainset}_{2r} = \{ p : \exists q \in \Sigma[\x]_{2r} \text{ with } p(\x) = q(\x) \text{ for all } \x \in \mainset\}
\end{equation}
for each $r \in \N$ when $\mainset = S^{n-1}$ or $\mainset = \{0,1\}^n$. Furthermore, as $B^n$ is defined by only a single constraint, we have $\quadmodule{B^n}_r = \preordering{B^n}_r$ for all $r \in \N$. Therefore, the Putinar- and Schm\"udgen-type hierarchies \eqref{EQ:lowbound} and \eqref{EQ:lowwbound} coincide in the case of the unit ball.

\subsection{Gegenbauer polynomials.} \label{SEC:gegenbauer}
We introduce now the \emph{Gegenbauer polynomials} (also known as \emph{ultraspherical} polynomials). We refer to the book of Szeg\"o~\cite{Szego:book} for a comprehensive treatment of univariate orthogonal polynomials.
For $n \geq 2$, let $\gegenweight{n}(x) := c_n (1-x^2)^{\frac{n-2}{2}}$, where $c_n > 0$ is chosen so that:
\[
\int_{-1}^1 \gegenweight{n}(x) dx = 1.
\]
The Gegenbauer polynomials $\{ \gegen{n}_k : k \in \N \}$ are defined as the set of orthogonal polynomials on $[-1,1]$ w.r.t. the weight function $\gegenweight{n}$. That is, the polynomial $\gegen{n}_k$ is of exact degree $k$ for each $k \in \N$, and the following orthogonality relations hold:
\begin{equation}
\label{EQ:gegendef}
\int_{-1}^1 \gegen{n}_k(x) \gegen{n}_{k'}(x) \gegenweight{n}(x) dx = 0 \quad (k \neq k').
\end{equation}
We adopt a normalization for which:
\begin{equation} \label{EQ:gegennormalization}
	\gegen{n}_k(x) = \gegenorth{n}_k(1) \gegenorth{n}_k(x) \quad (x \in \R),
\end{equation}
where $\gegenorth{n}_k$ is the ortho\emph{normal} Gegenbauer polynomial of degree $k$, i.e., satisfying:
\begin{equation*}
\int_{-1}^1 \gegenorth{n}_k(x) \gegenorth{n}_{k'}(x) \gegenweight{n}(x) dx = \delta_{kk'} \quad (k, k' \in \N).
\end{equation*}
We have that $\max_{x \in [-1, 1]} |\gegen{n}_k(x)| = \gegen{n}_k(1)$, and so it will also be convenient to write:
\[
\gegensup{n}_k(x) := \frac{\gegen{n}_k(x)}{\gegen{n}_k(1)} = \frac{\gegenorth{n}_k(x)}{\gegenorth{n}_k(1)} \quad (x \in \R)
\]
for the normalization of the Gegenbauer polynomials which satisfies: 
\[
	\max_{x \in [-1, 1]} |\gegensup{n}_k(x)| = \gegensup{n}_k(1) = 1 \quad (k \in \N).
\]
The upshot is that for these choices of normalization, we have:
\begin{equation} \label{EQ:gegennormalizations}
	\int_{-1}^1 \gegen{n}_k(x) \gegensup{n}_{k'}(x) \gegenweight{n}(x) dx = \int_{-1}^1 \gegenorth{n}_{k}(x) \gegenorth{n}_{k'}(x) \gegenweight{n}(x) dx = \delta_{kk'}.
\end{equation}
We will make use of the expansion:
\begin{equation} \label{EQ:qgegenexpansion}
	q(x) = \sum_{k=0}^{d} \lambda_k \gegen{n}_k(x) \quad (x \in \R)
\end{equation}
of a univariate polynomial $q$ of degree $d$ in the basis of Gegenbauer polynomials. Using \eqref{EQ:gegennormalizations}, the coefficients $\lambda_k$ in \eqref{EQ:qgegenexpansion} are given by:
\begin{equation} \label{EQ:lambdaintegral}
\lambda_k = \int_{-1}^1 \gegensup{n}_k(x) q(x) \gegenweight{n}(x)dx \quad (0 \leq k \leq d).
\end{equation}
Using \eqref{EQ:lambdaintegral} and the fact that $\gegensup{n}_k(x) \leq \gegensup{n}_0(x) = 1$ for every $x \in [-1, 1]$ and $k \in \N$, we find that:
\begin{equation}\label{EQ:lambdaleq1}
	\lambda_k \leq \lambda_0 \quad (1 \leq k \leq d)
\end{equation}
whenever $q$ is nonnegative on $[-1, 1]$. Furthermore, we note that if $q$ is an \emph{even} polynomial of degree $2d$, we may write:
\begin{equation} \label{EQ:evenqgegenexpansion}
q(x) = \sum_{k=0}^d \lambda_{2k} \gegen{n}_{2k}(x) \quad (x \in \R).
\end{equation}
Indeed, as the odd degree Gegenbauer polynomials are odd functions, the integral \eqref{EQ:lambdaintegral} vanishes for odd $k$ in this case.

\section{Construction of the linear operator} \label{SEC:CD}
In this section, we explain how to construct the linear operator $\kernelop$ of Section \ref{SEC:outline}. The starting point of this construction is the Christoffel-Darboux kernel.
\subsection{Closed forms of the Christoffel-Darboux kernel}
Let $\mainset \subseteq \R^n$ be a compact set and let $\mu$ be a measure whose support is exactly $\mainset$. We may define an inner product $\langle \cdot, \cdot \rangle_\mu$ on the space $\poly{\mainset}$ of polynomials on $\mainset$ by setting:
\[
	\langle p, q \rangle_\mu  := \int_{\mainset} p(\x) q(\x) d\mu(\x) \quad (p, q \in \poly{\mainset}).
\]
We write $\{ P_\alpha : \alpha \in \N^n \}$ for an orthonormal basis of $\poly{\mainset}$ w.r.t. this inner product, ordered so that $P_\alpha$ is of exact degree $|\alpha| = \sum_{i=1}^n \alpha_i$ for all $\alpha \in \N^n$. Writing $\N^n_r = {\{\alpha \in\N^n : |\alpha| \leq r\}}$, the \emph{Christoffel-Darboux} kernel $\kernelCD_r : \mainset \times \mainset \to \R$ of degree $r \in \N$ for $(\mainset, \mu)$ is then defined as:
\begin{equation}
\label{EQ:CDkernel}
\kernelCD_r(\x, \y) = \sum_{\alpha \in \N^n_r} P_{\alpha}(\x)P_{\alpha}(\y) \quad (\x, \y \in \mainset).
\end{equation} 
Using the inner product $\langle \cdot, \cdot \rangle_\mu$, any polynomial kernel $\kernel : \mainset \times \mainset \to \R$ induces a linear operator $\kernelop : \poly{\mainset} \to \poly{\mainset} $ by:
\[
	\kernelop p (\x) := \langle K(\x, \cdot), p \rangle_\mu =  \int_{\mainset} K(\x, \y) p(\y) d\mu(\y) \quad (\x \in \mainset,~p \in \poly{\mainset}).
\]
From the orthonormality of the $P_\alpha$, it follows that the operator $\kernelopCD_r$ associated to the Christoffel-Darboux kernel $\kernelCD_r$ acts as the identity on $\poly{\mainset}_r$. That is, we have:
\[
	\kernelopCD_r p(\x) = \int_{\mainset} \kernelCD_r(\x, \y) p(\y) d\mu(\y) = p(\x) \quad (\x \in \mainset,~p \in \poly{\mainset}_r).
\]
The Christoffel-Darboux kernel is therefore also called the \emph{reproducing kernel} for the inner product space $(\poly{\mainset}_r,\langle \cdot, \cdot \rangle_\mu$). 

For $k \in \N$, we write $H_k := \mathrm{span} \{ P_\alpha : |\alpha| = k \}$ for the subspace of $\poly{\mainset}$ spanned by the $P_\alpha$ of exact degree $k$. Note that we may equivalently define $H_k$ as:
\begin{equation}
\label{EQ:Vk}
	H_k := \{ p \in \poly{\mainset}_k : \langle p, q \rangle_\mu = 0 \text{ for all } q \in \poly{\mainset}_{k-1} \}.
\end{equation}
In particular, we see that $H_k$ does not depend on our choice of basis $\{ P_\alpha \}$, but only on the measure $\mu$. In light of this fact, it is convenient to adopt the vector-notation:
\[
	\mathbb{P}_k(\x) := (P_\alpha(\x))_{|\alpha| = k} \quad (k \in \N).
\]
The kernel $\kernelCDn{k}(\x, \y) := \mathbb{P}_k(\x)^\top \mathbb{P}_k(\y)$ does not depend on the choice of basis $\{ P_\alpha \}$, and its associated operator reproduces the subspace $H_k$. That is, if we decompose a polynomial $p \in \poly{\mainset}$ as:
\begin{equation} \label{EQ:degreedecomp}
	p(\x) = \sum_{k=0}^{\deg(p)} p_k(\x) \quad (p_k \in H_k),
\end{equation}
we then have that:
\begin{equation} \label{EQ:CDnreproducing}
	\kernelopCDn{k} p(\x) = \int_{\mainset} \mathbb{P}_k(\x)^\top \mathbb{P}_k(\y) p(\y) d\mu(\y) = p_k(\x) \quad (\x \in \mainset,~k \in \N).
\end{equation}
After regrouping the terms in \eqref{EQ:CDkernel}, we can express the Christoffel-Darboux kernel independently of the choice of basis as:
\begin{equation}
\label{EQ:CDkerneldecomp}
	\kernelCD_r(\x, \y) = \sum_{k = 0}^r \kernelCDn{k}(\x, \y) = \sum_{k = 0}^r \mathbb{P}_k(\x)^\top \mathbb{P}_k(\y).
\end{equation}
The expression \eqref{EQ:CDkernel} for the Christoffel-Darboux kernel is unwieldy, even when an explicit form of the orthogonal polynomials $P_\alpha$ is known. In certain special cases, a `closed form' expression for $\kernelCD_r$ may be derived, based on the regrouping \eqref{EQ:CDkerneldecomp}. \LS{Indeed, in these special cases, the term $\mathbb{P}_k(\x)^\top \mathbb{P}_k(\y)$ may be expressed by composing a \emph{univariate} polynomial with a relatively simple multivariate polynomial (e.g., the inner product $\x \cdot \y$ in \eqref{EQ:spheresummation})}. \LS{Such expressions will play a crucial role in the proofs of our main results, as they allow us to establish \eqref{PROPERTY:incone} (see also Appendix~\ref{APP:sosrepresentations})}.
We now present four special cases where these expressions are available: the hypersphere, the binary cube, the unit ball and the standard simplex. The treatment of the former two cases is classical, and the derived closed forms were already used in the context of showing degree bounds on positivity certificates in \cite{FangFawzi:sphere} and \cite{SlotLaurent:bincube}, respectively. For the treatment of the latter two cases we rely on the work of Xu \cite{Xu:ball, Xu:simplex}.

\subsubsection{The unit sphere \cite{FangFawzi:sphere}} Consider the unit sphere ${S^{n-1} := \{ \x \in \R^n : \|\x\|^2 = 1 \}}$, equipped with the $O(n)$-invariant probability measure. The space $\poly{S^{n-1}}$ of polynomials on $S^{n-1}$ is given by:
\[
	\poly{S^{n-1}} = \R[\x] ~ / ~ ( 1 - \| \x \|^2).
\]
It is a well-known fact that the spaces $H_k$ defined in \eqref{EQ:Vk} are given in this case by:
\[
	H_k = \mathrm{Harm}_k := \{ p \in \R[\x] : p \text{ is homogeneous of degree } k \text{ and \emph{harmonic}}\footnote{A polynomial $p \in \R[\x]$ is called harmonic if it is in the kernel of the Laplace operator $\nabla^2$, i.e., if $\sum_{i=1}^n \frac{\partial^2 p}{\partial x_i^2} = 0$.}
\},
\]
so that $\poly{S^{n-1}} = \bigoplus_{k=0}^\infty \mathrm{Harm}_k$. A polynomial $p \in \mathrm{Harm}_k$ is called a \emph{spherical harmonic} of degree $k$. We write $h_k := \mathrm{dim}(\mathrm{Harm}_k)$. If we choose a set of orthonormal spherical harmonics $\{ s_{k, j} : 1 \leq j \leq h_k\} \subseteq \mathrm{Harm}_k$ for each $k \in \N$, the Christoffel-Darboux kernel of degree $r \in \N$ is given by:
\[
	\kernelCD_r(\x, \y) = \sum_{k=0}^r \kernelCDn{k}(\x, \y) = \sum_{k=0}^r \sum_{j=1}^{h_k} s_{k, j} (\x) s_{k, j} (\y) \quad (\x, \y \in S^{n-1}).
\]
The spherical harmonics $s_{k, j}$ satisfy the following summation formula involving the Gegenbauer polynomials \eqref{EQ:gegendef}:
\begin{equation}
\label{EQ:spheresummation}
	\kernelCDn{k}(\x, \y) = \sum_{j=1}^{h_k} s_{k, j}(\x)s_{k, j}(\y) = \gegen{n-1}_k(\x \cdot \y) \quad (\x, \y \in S^{n-1}).
\end{equation}
Using \eqref{EQ:spheresummation}, we thus derive the following closed form of the Christoffel-Darboux kernel $\kernelCD_r$ on $S^{n-1}$, which is just the Funk-Hecke formula~\eqref{EQ:overviewex1}:
\[
	\kernelCD_r(\x, \y) = \sum_{k=0}^r \gegen{n-1}_k(\x \cdot \y) \quad (\x, \y \in S^{n-1}).
\]

\subsubsection{The binary cube \cite{SlotLaurent:bincube}}
\sloppy
Next, consider the $n$-dimensional binary cube ${\bincube{n} := \{0, 1\}^n \subseteq \R^n}$, which we equip with the uniform probability measure. The space of polynomials $\poly{\bincube{n}}$ on the binary cube is given by:
\[
	\poly{\bincube{n}} = \R[\x] / (x_i - x_i^2 : 1 \leq i \leq n) = \mathrm{span} \{ \x^\alpha : \alpha \in \{0, 1\}^n \}.
\]
We write $\chi_{\a} : \bincube{n} \to \R$ for the \emph{character} associated to a point $\a \in \bincube{n}$, defined by $\chi_{\a}(\x) = (-1)^{\a \cdot \x}$. Note that if $|\a| = k$, then $\chi_{\a} \in \poly{\bincube{n}}_k$. The characters form an orthogonal basis of $\poly{\bincube{n}}$, and the spaces $H_k$ are thus given in this case by $H_k = \mathrm{span} \{\chi_{\a} : |\a| = k \}$. The Christoffel-Darboux kernel of degree $r \leq n$ may then be written as:
\[
	\kernelCD_r(\x, \y) = \sum_{k = 0}^r \sum_{|\a| = k} \chi_{\a}(\x) \chi_{\a}(\y) \quad (\x, \y \in \bincube{n}).
\]
Similar to before, the characters of degree $k$ satisfy a summation formula, now in terms of the Krawtchouk polynomials\footnote{The Krawtchouk polynomials $\{ \kraw{n}_k : k \in \N \}$ are the orthogonal polynomials w.r.t. the discrete measure $\sum_{x=0}^n {n \choose x} \delta_x$. See, e.g., \cite{SlotLaurent:bincube, Szego:book}.}:
\begin{equation}
\label{EQ:charactersummation}
\kernelCDn{k}(\x, \y) = \sum_{|\a| = k} \chi_{\a}(\x) \chi_{\a}(\y) = \kraw{n}_k( |\x - \y| ) \quad (\x, \y \in \bincube{n}).
\end{equation}
We may use \eqref{EQ:charactersummation} to obtain the following closed form of $\kernelCD_r$ on $\{0,1\}^n$, which is just formula~\eqref{EQ:overviewex2}:
\[
\kernelCD_r(\x, \y) = \sum_{k=0}^r \kraw{n}_k(|\x - \y|) \quad (\x, \y \in \bincube{n}).
\]

\subsubsection{The unit ball}
Now consider the unit ball $\ball{n} := \{\x \in \R^n : \|\x\|^2 \leq 1 \} \subseteq \R^n$, which we equip with the $O(n)$-invariant probability measure $\mu_{\ball{}}$ given by:
\begin{equation} \label{EQ:ballmeasure}
d\mu_{\ball{}}(\x) = c_n (1 - \|\x\|^2)^{-\frac{1}{2}}d\x \quad (\x \in \ball{n}),
\end{equation}
with $c_n > 0$ a normalization constant. As $\ball{n}$ is full-dimensional, we have ${\poly{\ball{n}} = \R[\x]}$. Xu derives the following closed form of the Christoffel-Darboux kernel. 
\begin{theorem}[Xu \cite{Xu:ball}, Theorem 3.1]
Let $\{ P_\alpha : \alpha \in \N \}$ be an orthonormal basis of $\poly{\ball{n}}$ w.r.t. $\mu_{\ball{}}$. Then the $P_\alpha$ satisfy the following summation formula in terms of the Gegenbauer polynomials\footnote{Note that the Gegenbauer polynomial of degree $k$ in \cite{Xu:ball} differs by a factor $(k + \frac{n-1}{2})/\frac{n-1}{2}$ from the one used here (compare (2.10) in \cite{Xu:ball} to \eqref{EQ:gegennormalization}).} \eqref{EQ:gegendef}:
\begin{equation}
\label{EQ:ballsummation}
\begin{split}
\mathbb{P}_k(\x)^\top \mathbb{P}_k(\y) = \frac{1}{2} \cdot \big(&\gegen{n}_k(\x \cdot \y + \sqrt{1-\|\x\|^2}\sqrt{1-\|\y\|^2}) \quad +  \\
&\gegen{n}_k(\x \cdot \y - \sqrt{1-\|\x\|^2}\sqrt{1-\|\y\|^2})\big)
\end{split}
\quad\quad (\x, \y \in \ball{n}).
\end{equation}
Using \eqref{EQ:ballsummation}, we have the following closed form of the Christoffel-Darboux kernel $\kernelCD_r$:
\[
\begin{split}
\kernelCD_r(\x, \y) = \frac{1}{2} \sum_{k=0}^r \big(&\gegen{n}_k(\x \cdot \y + \sqrt{1-\|\x\|^2}\sqrt{1-\|\y\|^2}) \quad +  \\
&\gegen{n}_k(\x \cdot \y - \sqrt{1-\|\x\|^2}\sqrt{1-\|\y\|^2})\big)
\end{split}
\quad\quad (\x, \y \in \ball{n}).
\]
\end{theorem}

\subsubsection{The standard simplex}
Finally, consider the standard simplex:
\[
\simplex{n} := \{\x\in\R^n: \x \geq 0, ~1-|\x|\geq 0\} \subseteq \R^n.
\]
We equip $\simplex{n}$ with the  probability measure $\mu_{\simplex{}}$ given by:
\begin{equation} \label{EQ:simplexmeasure}
	d\mu_\simplex{}(\x) = c_n x_1^{- 1/2}x_2^{- 1/2} \ldots x_n^{- 1/2} (1 - |\x|)^{ - 1/2} d\x \quad (\x \in \simplex{n}),
\end{equation}
where $c_n > 0$ is a normalization constant. As the simplex is full-dimensional, we have $\poly{\simplex{n}} = \R[\x]$.
Xu derives the following closed form of the Christoffel-Darboux kernel. 
\begin{theorem}[Xu \cite{Xu:simplex}, Corollary 2.4]
Let $\{ P_\alpha : \alpha \in \N \}$ be an orthonormal basis of $\poly{\simplex{n}}$ w.r.t. $\mu_\simplex{}$. Then the $P_\alpha$ satisfy the following summation formula in terms of the Gegenbauer polynomials \eqref{EQ:gegendef}:
\begin{equation}
\label{EQ:simplexsummation}
\mathbb{P}_k(\x)^\top \mathbb{P}_k(\y) = \frac{1}{2^{n+1}} \sum_{t \in \{-1, 1\}^{n+1}} \gegen{n}_{2k} \big(\sum_{i=1}^{n+1} \sqrt{x_iy_i}t_i\big) \quad (\x, \y \in \simplex{n}).
\end{equation}
Here, we write $x_{n+1} := 1-|\x|$, $y_{n+1} := 1-|\y|$.
Using \eqref{EQ:simplexsummation}, we have the following closed form of the Christoffel-Darboux kernel $\kernelCD_r$:
\[
\kernelCD_r(\x, \y) = \frac{1}{2^{n+1}} \sum_{k=0}^r \quad \sum_{t \in \{-1, 1\}^{n+1}} \gegen{n}_{2k} \big(\sum_{i=1}^{n+1} \sqrt{x_iy_i}t_i\big) \quad (\x, \y \in \simplex{n}).
\]
\end{theorem}

\subsection{Sum-of-squares representations}
Based on the closed forms of the Christoffel-Darboux kernel derived above, we may define kernels ${\kernel(\x,\y) = \kernelCDp_r(\x, \y; \lambda)}$ that satisfy property \eqref{PROPERTY:incone}. 
Let us first prove a claim made in Section~\ref{SEC:outline}; namely that it suffices for \eqref{PROPERTY:incone} to hold that $\x \mapsto \kernel(\x, \y)$ lies $\preordering{\mainset}_{2r}$ for all $\y \in \mainset$ fixed.
\begin{lemma} \label{LEM:incone}
Let $\mainset \subseteq \R^n$ be a compact semialgebraic set, and let $\mu$ be a finite measure supported on $\mainset$. Let $Q \subseteq \R[\x]$ be a convex cone, and suppose that ${\kernel : \mainset \times \mainset \to \R}$ is a polynomial kernel for which $\kernel(\cdot, \y) \in Q$ for each $\y \in \mainset$ fixed. Then if $p \in \R[\x]$ is nonnegative on $\mainset$, we have $\kernelop p \in Q$. That is, when selecting $Q = \preordering{\mainset}_{2r}$, the operator $\kernelop$ associated to $\kernel$ satisfies \eqref{PROPERTY:incone}.
\end{lemma}
\begin{proof}
Let $\{ (\y_i, w_i) : 1 \leq i \leq N \} \subseteq \mainset \times \R_{>0}$ be a cubature rule for the integration of polynomials of degree up to $\mathrm{deg}(p) + \mathrm{deg}(\kernel)$ over $\mainset$ w.r.t. the measure $\mu$, whose existence is guaranteed by  Tchakaloff's Theorem \cite{Tchakaloff} (see also \cite{deKlerkLaurent:survey}). Then by definition, we have:
\[
	\kernelop p (\x) = \int_{\mainset} \kernel(\x, \y) p(\y) d\mu(\y) = \sum_{i=1}^N \kernel(\x, \y_i) w_i p(\y_i) \quad (\x \in \mainset).
\]
As $w_ip(\y_i) \geq 0$ and $\kernel(\cdot, \y_i) \in Q$ for all $1 \leq i \leq N$, this shows that $\kernelop p \in Q$.
\end{proof}
Based on Lemma~\ref{LEM:incone}, relation \eqref{EQ:spherebincubemodule} and the closed forms of the Christoffel-Darboux kernel derived above, we see immediately that the kernels of Example~\ref{EX:sphere-cube} for the hypersphere and binary cube indeed satisfy \eqref{PROPERTY:incone}. Turning now to the unit ball and simplex, the situation is slightly more complicated.
\subsubsection{The unit ball} Let $q \in \Sigma[x]_{2r}$ be a univariate sum of squares, with expansion $q(x) = \sum_{k=0}^{2r}\lambda_k\gegen{n}_k(x)$ in the basis of Gegenbauer polynomials \eqref{EQ:qgegenexpansion}. In light of the closed form \eqref{EQ:ballsummation} of the Christoffel-Darboux kernel on the unit ball, we have:
\begin{equation} \label{EQ:ballqkernel}
\begin{split}
\kernelCDp_{2r}(\x, \y; \lambda) = 
\frac{1}{2} \big( &q(\x \cdot \y + \sqrt{1-\|\x\|^2}\sqrt{1-\|\y\|^2}) \quad +  \\
&q(\x \cdot \y - \sqrt{1-\|\x\|^2}\sqrt{1-\|\y\|^2}) \big)
\end{split}
\quad\quad (\x, \y \in \ball{n}).
\end{equation}

\begin{lemma} \label{LEM:ballsosrepresentation}
Let $q \in \Sigma[x]_{2r}$ be a univariate sum of squares. Then the kernel $\kernelCDp_{2r}(\x, \y; \lambda)$ in \eqref{EQ:ballqkernel} satisfies $\kernelCDp_{2r}(\cdot, \y; \lambda) \in \preordering{\ball{n}}_{2r}$ for fixed $\y \in \ball{n}$. As a result, its associated operator satisfies \eqref{PROPERTY:incone} by Lemma~\ref{LEM:incone}.
\end{lemma}\noindent We postpone the proof to Appendix \ref{APP:sosrepresentations}.

\subsubsection{The simplex}
Let $q(x) = \sum_{k=0}^{4r}\lambda_k\gegen{n}_k(x)$ again be a univariate sum of squares, now of degree $4r$. Using \eqref{EQ:evenqgegenexpansion}, we have:
\begin{equation} \label{EQ:evenq}
q_{\rm even}(x) := \frac{q(x) + q(-x)}{2} = \sum_{k=0}^{2r}\lambda_{2k}\gegen{n}_{2k}(x) \quad (x \in \R).
\end{equation}	
In light of the closed form \eqref{EQ:simplexsummation} of the Christoffel-Darboux kernel on the simplex, we find: 
\begin{equation} \label{EQ:simplexqkernel}
\kernelCDp_{2r}(\x, \y; \lambda_{\rm even}) = \frac{1}{2^{n+1}} \sum_{t \in \{-1, 1\}^{n+1}} q_{\rm even}\big(\sum_{i=1}^{n+1} \sqrt{x_iy_i}t_i\big) \quad (\x, \y \in \simplex{n}).
\end{equation}
Here, $\lambda_{\rm even} := (\lambda_{2k})_{0 \leq k \leq 2r}$ and $x_{n+1} = 1 - |\x|$, $y_{n+1} = 1 - |\y|$.

\begin{lemma} \label{LEM:simplexsosrepresentation}
Let $q \in \Sigma[x]_{4r}$ be a univariate sum of squares of degree $4r$, and let $q_{\rm even}$ be as in \eqref{EQ:evenq}. Then the kernel $\kernelCDp_{2r}(\x, \y; \lambda_{\rm even})$ in \eqref{EQ:simplexqkernel} satisfies $\kernelCDp_{2r}(\cdot, \y; \lambda_{\rm even}) \in \preordering{\simplex{n}}_{2r}$ for fixed $\y \in \simplex{n}$. As a result, its associated operator satisfies \eqref{PROPERTY:incone} by Lemma~\ref{LEM:incone}.
\end{lemma}\noindent We postpone the proof to Appendix \ref{APP:sosrepresentations}.

\newcommand{\harmbound}[1]{\gamma(#1)}
\section{Analysis of the linear operator} \label{SEC:P2}


Let $\kernelop$ be the operator associated to the perturbed Christoffel-Darboux kernel ${\kernel(\x, \y) := \kernelCDp_{2r}(\x, \y; \lambda)}$ defined in~\eqref{EQ:overvieweq1} of degree $2r$ for certain $\lambda = (\lambda_k)_{0 \leq k \leq 2r}$. 
Recall from \eqref{EQ:CDnreproducing} that for any polynomial $p$ on $\mainset$ of degree $d$, we have ${\kernelop p = \sum_{k=0}^d \lambda_k p_k}$, where ${p = \sum_{k=0}^d p_k}$ is the decomposition of \eqref{EQ:degreedecomp}. We see immediately that $\kernelop$ satisfies \eqref{PROPERTY:normalization} if and only if $\lambda_0 = 1$ (i.e., when $\kernelop(1) = \lambda_0 = 1$). 
We now expand on the second claim made in Section~\ref{SEC:outline}; namely that \eqref{PROPERTY:infnorm} may be shown for the operator $\kernelop$  by analyzing the difference between the coefficients $\lambda_k$ and $1$.

Recall that we consider in \eqref{PROPERTY:infnorm} a polynomial $f$ on $\mainset$ of degree $d$, whose sup-norm $\infnorm{\mainset}{f}$ over $\mainset$ is at most $1$ by assumption, and that we wish to bound ${\infnorm{\mainset}{\kernelop^{-1}f - f}}$.
Assuming that $\lambda_0 = 1$ and $\lambda_k \neq 0$ for $1 \leq k \leq d$, we have $\kernelop^{-1} f = \sum_{k=0}^d (1 - 1/ \lambda_k) f_k$ and so:
\[
	\infnorm{\mainset}{\kernelop^{-1} f - f} = \infnorm{\mainset}{ \sum_{k=1}^d (1 - 1/\lambda_k) f_k } \leq \max_{1 \leq k \leq d} \infnorm{\mainset} {f_k} \cdot \sum_{k=1}^d |1 - 1/\lambda_k |.
\]
We have shown the following.
\begin{lemma} \label{LEM:P1P3}
Let $\kernelop$ be the operator associated to the perturbed Christoffel-Darboux kernel ${\kernel(\x, \y) := \kernelCDp_{2r}(\x, \y; \lambda)}$ defined in~\eqref{EQ:overvieweq1} of degree $2r$ for certain $\lambda = (\lambda_k)_{0 \leq k \leq r}$. Then $\kernelop$ satisfies \eqref{PROPERTY:normalization} if $\lambda_0=1$, and it satisfies \eqref{PROPERTY:infnorm} with:
\begin{equation} \label{EQ:infnormbound}
	\epsilon = \max_{1 \leq k \leq d} \infnorm{\mainset} {f_k} \cdot \sum_{k=1}^d |1 - 1/\lambda_k |.
\end{equation}
\end{lemma}
We now work to bound the quantity \eqref{EQ:infnormbound}.
\subsection{The harmonic constant} \label{SEC:harmbound}
In light of the factor $\max_{1 \leq k \leq d} \infnorm{\mainset} {f_k}$ in \eqref{EQ:infnormbound}, we define the parameter:
\begin{equation} \label{EQ:harmbounddef}
	\harmbound{\mainset}_d := \max_{p \in \R[\x]_d}\max_{0 \leq k \leq d} \frac{\infnorm{\mainset}{p_k}}{\infnorm{\mainset}{p}}.
\end{equation}
for any compact semialgebraic set $\mainset$ (equipped with a measure $\mu$).
Note that $\max_{1 \leq k \leq d} \infnorm{\mainset} {f_k} \leq \harmbound{\mainset}_d$ by definition. Let us first remark that $\harmbound{\mainset}_d < \infty$. Indeed, this follows immediately from equivalence of norms on the finite-dimensional vector space $\poly{\mainset}_d$. In particular, $\harmbound{\ball{n}}_d$ and $\harmbound{\simplex{n}}_d$ are finite constants depending only on $n$ and $d$.

For special cases of $\mainset$, more can be shown. For instance, when $\mainset$ is the hypersphere or binary cube, the constant $\harmbound{\mainset}_d$ may be bounded \emph{independently} of the dimension $n$ \cite{FangFawzi:sphere, SlotLaurent:bincube}. For the unit ball and simplex, we have the following.
\begin{proposition} \label{PROP:harmbound}
Let $\harmbound{\ball{n}}_d$ and $\harmbound{\simplex{n}}_d$ be the constants of \eqref{EQ:harmbounddef} on the unit ball and simplex, respectively. Then $\harmbound{\ball{n}}_d$ and $\harmbound{\simplex{n}}_d$ depend polynomially on $n$ (when $d$ is fixed) and polynomially on $d$ (when $n$ is fixed).
\end{proposition}
\begin{proof}
Let $\mainset = \ball{n}, \simplex{n}$, equipped with the probability measure $\mu_{\mainset} = \mu_{\ball{}}$, $\mu_{\simplex{}}$, respectively. Let $p \in \R[\x]_d$ be a polynomial of degree $d$ and assume that $\infnorm{\mainset}{p} = 1$. For $k \leq d$, we know from \eqref{EQ:CDnreproducing} that:
\[
	p_k(\x) = \kernelopCDn{k} p(\x) = \int_{\mainset} \kernelCDn{k}(\x, \y) p(\y) d\mu_{\mainset}(\y) \quad (\x \in \mainset).
\]
Using the fact that $\infnorm{\mainset}{p} = 1$ and $\mu_{\mainset}$ is a probability measure, as well as the Cauchy-Schwarz inequality, we find that:
\begin{align*}
\begin{split}
	|p_k(\x)|^2 &= |\int_{\mainset} \kernelCDn{k}(\x, \y) p(\y) d\mu_{\mainset}(\y)|^2 \\
	&\leq \int_{\mainset} \kernelCDn{k}(\x, \y)^2 d\mu_{\mainset}(\y) \cdot \int_{\mainset} p(\y)^2 d\mu_{\mainset}(\y) \leq \int_{\mainset} \kernelCDn{k}(\x, \y)^2 d\mu_{\mainset}(\y)
\end{split}
\quad (\x \in \mainset).
\end{align*}
Using \eqref{EQ:CDnreproducing} again, we have:
\begin{equation} \label{EQ:Christfunc}
\int_{\mainset} \kernelCDn{k}(\x, \y)^2 d\mu_{\mainset}(\y) = \kernelCDn{k}(\x, \x) \quad (\x \in \mainset).
\end{equation}
It follows that:
\[
	\harmbound{\mainset}^2_d \leq \max_{0 \leq k \leq d }\max_{\x \in \mainset} \kernelCDn{k}(\x, \x).
\]
The closed forms \eqref{EQ:ballsummation} and \eqref{EQ:simplexsummation} of $\kernelCDn{k}$ allow us to bound \eqref{EQ:Christfunc}. On the ball, we have:
\[
\kernelCDn{k}(\x, \x) = \frac{1}{2} \cdot \big(\gegen{n}_k(1) + \gegen{n}_k(2\|\x\|^2 - 1)\big) \quad (\x \in \ball{n}).
\]
In particular, we have:
\[
	\harmbound{\ball{n}}_d^2 \leq \max_{\x \in \ball{n}} \kernelCDn{k}(\x, \x) \leq \max_{-1 \leq x \leq 1} |\gegen{n}_k(x)| = \max_{0 \leq k \leq d} \gegen{n}_k(1).
\]
On the simplex, we similarly have:
\[
	\kernelCDn{k}(\x, \x) =  \frac{1}{2^{n+1}} \sum_{t \in \{-1, 1\}^{n+1}} \gegen{n}_{2k} \big(\sum_{i=1}^{n+1} x_it_i\big) \quad (\x \in \simplex{n})
\]
and so:
\[
	\harmbound{\simplex{n}}_d^2 \leq \max_{\x \in \simplex{n}} \kernelCDn{k}(\x, \x) \leq \max_{-1 \leq x \leq 1} |\gegen{n}_{2k}(x)| = \max_{0 \leq k \leq d} \gegen{n}_{2k}(1).
\]
Finally, we note that (see, e.g., (2.9) in \cite{Xu:ball}):
\[
	\gegen{n}_k(1) = (1 + \frac{2k}{n-1}) \cdot {k + n - 2 \choose k} \quad\quad (n, k \in \N).
\]
We conclude that the constant $\harmbound{\ball{n}}_d$ satisfies:
\[
	\harmbound{\ball{n}}^2_d \leq \max_{0 \leq k \leq d} \gegen{n}_k(1) = \max_{0 \leq k \leq d} (1 + \frac{2k}{n-1}) \cdot {k + n - 2 \choose k}.
\]
The constant $\harmbound{\simplex{n}}_d$ similarly satisfies:
\[
	\harmbound{\simplex{n}}^2_d \leq \max_{0 \leq k \leq d} \gegen{n}_{2k}(1) = \max_{0 \leq k \leq d} (1 + \frac{4k}{n-1}) \cdot {2k + n - 2 \choose 2k}.
\]
\end{proof}

\subsection{Selecting a univariate square}


The final ingredient we need for the proof of our main theorems is the following result, due essentially to Fang and Fawzi \cite{FangFawzi:sphere}.
\begin{lemma}[see \cite{FangFawzi:sphere}, Theorem 6]\label{LEM:qlambdabound}
Let $n, d \in \N$. Then for every $r \geq 2nd\sqrt{d}$ there exists a univariate sum of squares $q(x) = \sum_{k=0}^{2r} \lambda_k \gegen{n}_k(x)$ of degree $2r$ with $\lambda_0 = 1$ and:
\[
	\sum_{k=1}^d |1 - 1/\lambda_k | \leq \frac{2n^2d^3}{r^2}.
\]
\end{lemma}
We outline how to obtain this result here, following the strategy of \cite{FangFawzi:sphere}. We emphasize a connection to the upper bounds \eqref{EQ:upbound} in a univariate setting which was only implicitely present in \cite{FangFawzi:sphere}. We also state the intermediary result Lemma~\ref{LEM:qanalysis}, which we need to prove Theorem~\ref{THM:secondball} and Theorem~\ref{THM:secondsimplex} in Section~\ref{SEC:upbounds}.
The first step of the argument is to linearize the quantity $\sum_{k=1}^d |1 - 1/\lambda_k |$.
\begin{lemma} \label{LEM:qlinearize}
Let $n, d, r \in \N$ and let $q(x) = \sum_{k=0}^{2r} \lambda_k \gegen{n}_k(x)$ be a sum of squares. Assuming that $\lambda_0 = 1$ and $\lambda_k \geq 1/2$ for $1 \leq k \leq d$ we have:
\[
	\sum_{k=1}^d |1 - 1/\lambda_k | \leq 2 \sum_{k=1}^d (1 - \lambda_k)
\]
\end{lemma}
\begin{proof}
As $q$ is nonnegative on $[-1, 1]$, we know that $\lambda_k \leq \lambda_0 = 1$ by \eqref{EQ:lambdaleq1}. As $\lambda_k \geq 1/2$ for each $k$, we have:
\[
\sum_{k=1}^d |1 - 1/\lambda_k | = \sum_{k=1}^d \frac{|1 - \lambda_k |}{\lambda_k} \geq 2 \sum_{k=1}^d |1-\lambda_k| = 2\sum_{k=1}^d (1-\lambda_k).
\]
\end{proof}
It remains to choose a sum of squares $q$ minimizing $\sum_{k=1}^d (1-\lambda_k)$. This turns out to reduce to analyzing a univariate instance of the upper bounds \eqref{EQ:upbound}.
\begin{lemma} \label{LEM:qanalysis}
Let $n, d \in \N$. Then for every $r \geq d$ there exists a univariate sum of squares $q(x) = \sum_{k=0}^{2r} \lambda_k \gegen{n}_k(x)$ of degree $2r$ with $\lambda_0 = 1$ and:
\[
	\sum_{k=1}^d (1-\lambda_k) \leq \frac{n^2d^3}{r^2}.
\]
\end{lemma}
\begin{proof}
For a univariate sum of squares $q(x) = \sum_{k=0}^{2r} \lambda_k \gegen{n}_k(x)$, the coefficients $\lambda_k$ are equal to (see  relation \eqref{EQ:lambdaintegral}):
\[
	\lambda_k = \int_{-1}^1  \gegensup{n}_k(x) q(x) \gegenweight{n}(x) dx \quad (0 \leq k \leq d),
\]
where $\gegensup{n}_k$ is the normalization of the Gegenbauer polynomial of degree $k$ satisfying $\max_{-1 \leq x \leq 1} |\gegensup{n}_k(x)| = \gegensup{n}_k(1) = 1$. 
Let $h(x) := d - \sum_{k=1}^d \gegensup{n}_k(x)$. Note that:
\[
\int_{-1}^1 h(x) q(x) \gegenweight{n}(x)dx = d-\sum_{k=1}^d \lambda_k.
\]
Selecting $q$ optimally is thus equivalent to solving the optimization problem:
\begin{equation} \label{EQ:qupboundprogram}
	\mathrm{opt} := \inf_{q \in \Sigma[x]_{2r}} \left\{ \int_{-1}^1  h(x) q(x) \gegenweight{n}(x) dx : \int_{-1}^1 q(x) \gegenweight{n}(x) dx = 1 \right\}.
\end{equation}
We recognize the program \eqref{EQ:qupboundprogram} as
the Lasserre-type upper bound $h^{(r)}$ (see \eqref{EQ:upbound}) on the minimum $h_{\min} = h(1) = 0$ on $[-1, 1]$ w.r.t. the measure $d\mu(x) = \gegenweight{n}(x)dx$.

The behaviour of the upper bounds in this univariate setting is well-understood. For instance, it is known that for the \emph{linear} polynomial $u(x) = 1-x$, the bounds have error $u^{(r)} - u_{\min} = (\xi_{r+1} + 1)$, where $\xi_{r+1}$ is the smallest root of the Gegenbauer polynomial $\gegen{n}_{r+1}$ of degree $r + 1$ \cite{deKlerkLaurent:hypercube2020}. These roots satisfy
\begin{equation} \label{EQ:rootbound}
0 \leq (\xi_{r+1}+1) \leq \frac{1}{4} \cdot \frac{n^2}{r^2},
\end{equation}
see~\cite[Section 2.3]{Jordaan:roots}, and also~\cite[Proposition 7]{FangFawzi:sphere}.
The error $h^{(r)} - h_{\min}$ may then be bounded by considering the linear Taylor estimate of $h$ at $1$. For this, note first that by~\cite[Proposition 18]{FangFawzi:sphere}, we have 
\[
	\gegensup{n}_k(x) \geq \big(\gegensup{n}_k\big)'(1) \cdot (x-1) + \gegensup{n}_k(1) \quad \forall k \geq 0, \, x \in [-1, 1],
\]
and so we find that
\[
	h(x) \leq -h'(1) (1-x) \quad x \in [-1, 1].
\]
Using~\eqref{EQ:rootbound}, it follows that the optimum value of \eqref{EQ:qupboundprogram} is at most
\[
	\mathrm{opt} \leq -(\xi_{r+1}+1)\cdot h'(1) \leq \frac{1}{4} \cdot \frac{-h'(1) \cdot n^2}{r^2}.
\]
It remains to bound $-h'(1)$. This requires some computations involving standard identities for Gegenbauer polynomials, for which we follow~\cite{Surey:thesis}.
For $k \in \N$, $\lambda = \frac{n-1}{2}$, let $P_k^{(\lambda)}$ be the ultraspherical/Gegenbauer polynomial as defined in~\cite[Section 4.7]{Szego:book}, and note that $\gegen{n}_k = \frac{k + \lambda}{\lambda} \cdot P_k^{(\lambda)}$. It holds~\cite[(4.7.3), (4.7.27)]{Szego:book} that 
\[P_k^{(\lambda)}(1) = {k + 2\lambda -1 \choose k}, \text{ and }
\big(P_k^{(\lambda)}\big)'(x) = 2 \lambda P_{k-1}^{(\lambda + 1)}(x) \quad \forall x \in \R.
\]
Now we may compute
\begin{align*}
	\big(\gegensup{n}_k\big)'(1) &= \frac{\big(\gegen{n}_k\big)'(1)}{\gegen{n}_k(1)} = \frac{\big(P^{(\lambda)}_k\big)'(1)}{P^{(\lambda)}_k(1)} \\ 
	&= 2\lambda \cdot {k-1 + 2\lambda + 1 \choose k -1} \cdot {k+2\lambda-1 \choose k}^{-1} \\
	&= 2 \lambda \cdot \frac{k(k+2\lambda)}{2\lambda(2\lambda+1)}
	= \frac{(n-1)k(k+n-1)}{(n-1)n}	= \frac{k(k+n-1)}{n}.
\end{align*}
Finally, we get that
\[
	-h'(1) = \sum_{k=1}^d \big(\gegensup{n}_k\big)'(1) = \frac{d(d+1)(2d+3n-2)}{6n} \leq 4d^3. \qedhere
\]
\end{proof}
\begin{remark}
In an earlier version of this paper, an incorrect version of Lemma~\ref{LEM:qanalysis} claimed a bound on $\sum_{k=1}^d (1-\lambda_k)$ with quadratic (rather than cubic) dependence on $d$. 
This mistake was pointed out to the author by Alison Surey and Markus Schweighofer. Surey presents a proof of Lemma~\ref{LEM:qanalysis} in their master's thesis~\cite{Surey:thesis}, which is unpublished but made available to the author. The final part of the proof presented here loosely follows~\cite{Surey:thesis}. It should be noted that a quadratic dependence on $d$ is obtained in~\cite[Theorem 6]{FangFawzi:sphere}, \emph{under the assumption that} $n \geq d$. This would hold in our setting as well (under the same assumption).
\end{remark}

\subsection{Proof of Theorem~\ref{THM:mainball} and Theorem~\ref{THM:mainsimplex}}
We have now gathered all tools required to prove our main results. First, let $\mainset = \ball{n}$ and let $f$ be a polynomial on $\mainset$ of degree $d$. Recall that we may assume w.l.o.g. that $\funcmin = 0$ and that $\infnorm{\mainset}{f} = 1$. We show how to construct an operator $\kernelop$ satisfying the properties \eqref{PROPERTY:normalization}, \eqref{PROPERTY:incone} and \eqref{PROPERTY:infnorm} for appropriate $\epsilon > 0$, whose existence will immediately imply Theorem~\ref{THM:mainball} by Lemma~\ref{LEM:overviewsummary}. See also Figure~\ref{FIG:schematic}.

For $r \geq 2d\sqrt{d}$, we select a univariate sum of squares:
\[q(x) = \sum_{k=0}^{2r} \lambda_k \gegen{n}_k(x)\]
as in Lemma~\ref{LEM:qlambdabound}, i.e., such that ${\lambda_0 = 1}$ and ${\sum_{k=1}^d |1 - 1/\lambda_k|}$ is small. Consider the kernel $\kernel(\x, \y) := \kernelCDp_{2r}(\x, \y; \lambda)$ of \eqref{EQ:ballqkernel}. By Lemma~\ref{LEM:ballsosrepresentation}, we know  that the operator $\kernelop$ associated to $\kernel$ satisfies \eqref{PROPERTY:incone}. Furthermore, Lemma~\ref{LEM:P1P3} tell us that $\kernelop$~satisfies \eqref{PROPERTY:normalization} and that it satisfies \eqref{PROPERTY:infnorm} with:
\[
	\epsilon = \max_{1 \leq k \leq d} \infnorm{\mainset} {f_k} \cdot \sum_{k=1}^d |1-1/\lambda_k| \leq \harmbound{\ball{n}}_d \cdot \frac{2n^2d^3}{r^2}.
\]
Here, we use \eqref{EQ:harmbounddef} and Lemma~\ref{LEM:qlambdabound} for the inequality. We may thus apply Lemma~\ref{LEM:overviewsummary} to conclude the statement of Theorem~\ref{THM:mainball} with constant:
\begin{equation} \label{EQ:ballconstant}
	C_B(n, d) = 2n^2d^3 \harmbound{\ball{n}}_d.
\end{equation}
In light of Proposition~\ref{PROP:harmbound}, this constant has the promised polynomial dependence on $n$ (for fixed $d$) and on $d$ (for fixed $n$).

The proof of Theorem~\ref{THM:mainsimplex} for $\mainset = \simplex{n}$ is nearly identical. The only difference is that we should now select a sum of squares $q(x) = \sum_{k=0}^{4r} \lambda_k \gegen{n}_k(x)$ of degree $4r$ by applying Lemma~\ref{LEM:qlambdabound} for $d \gets 2d, r \gets 2r$ and consider the kernel $\kernel$ defined in \eqref{EQ:simplexqkernel}. The associated operator satisfies \eqref{PROPERTY:incone} by Lemma~\ref{LEM:simplexsosrepresentation}. By Lemma~\ref{LEM:P1P3}, it satisfies \eqref{PROPERTY:normalization} and \eqref{PROPERTY:infnorm} with:
\[
	\epsilon \leq \harmbound{\simplex{n}}_d \cdot \frac{2n^2d^3}{(2r)^2} = \harmbound{\simplex{n}}_d \cdot \frac{2n^2d^3}{r^2},
\]
using \eqref{EQ:infnormbound}, Lemma~\ref{LEM:qlambdabound} and \eqref{EQ:evenq}. We may apply Lemma~\ref{LEM:overviewsummary} to conclude the statement of Theorem~\ref{THM:mainsimplex} with:
\begin{equation} \label{EQ:simplexconstant}
	C_{\simplex{}}(n, d) = 2n^2d^3 \harmbound{\simplex{n}}_d,
\end{equation}
which also has the required dependencies on $n$ and $d$ by Proposition~\ref{PROP:harmbound}.
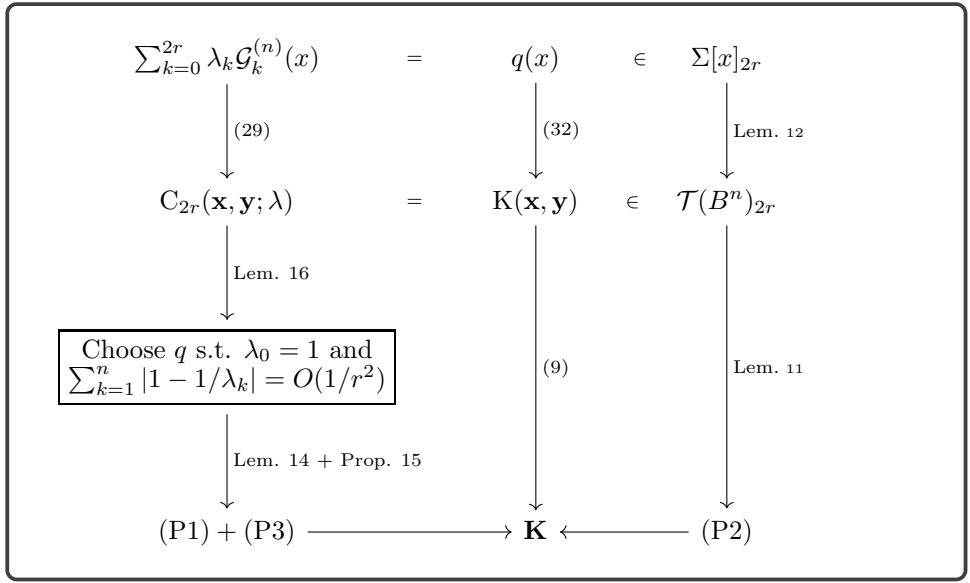
\begin{figure}
\centering
\begin{tcolorbox}[colback=white]
\begin{tikzcd}[contains/.style = {draw=none,"\in" description,sloped}, isequal/.style = {draw=none,"=" 
description,sloped}, column sep=normal, row sep=huge]
\makebox[9em][c]{$\sum_{k=0}^{2r} \lambda_k \gegen{n}_k(x)$} \ar[r, isequal] \ar[d, "\tiny \eqref{EQ:ballsummation}"] &
\makebox[4em][c]{$q(x)$} \ar[r, contains] \arrow[d, "\tiny \eqref{EQ:ballqkernel}"] 
& \Sigma[x]_{2r} \ar[d, "\tiny \text{Lem.}~\ref{LEM:ballsosrepresentation}"] \\
\makebox[9em][c]{$\kernelCDp_{2r}(\x, \y; \lambda)$}  \ar[r, isequal] \ar[d, "\text{Lem.}~\ref{LEM:qlambdabound}"]
& \makebox[4em][c]{$\kernel(\x,\y)$} \ar[r, contains] \ar[dd, "\eqref{EQ:associatedkernel}"]
& \preordering{\ball{n}}_{2r} \ar[dd, "\tiny \text{Lem.}~\ref{LEM:incone}"] \\
{ \let\scriptstyle\textstyle \fbox{$\substack{\text{Choose } q \text{ s.t. }\lambda_0=1 \text{ and} \\ \sum_{k=1}^n |1 - 1/\lambda_k| = O(1/r^2)}$}}
\ar[d, "\text{Lem.}~\ref{LEM:P1P3}\text{ + }\text{Prop.}~\ref{PROP:harmbound}"]
&  &  \\
\eqref{PROPERTY:normalization} + \eqref{PROPERTY:infnorm}
& \kernelop 
\ar[l, leftarrow] \ar[r, leftarrow]
& \eqref{PROPERTY:incone}
\end{tikzcd}
\end{tcolorbox}
\caption{Overview of the construction of a a linear operator ${\kernelop : \R[\x]_d \to \R[\x]_d}$ satisfying the properties \eqref{PROPERTY:normalization}, \eqref{PROPERTY:incone}, \eqref{PROPERTY:infnorm} of Section~\ref{SEC:outline} for the unit ball. The construction for the standard simplex is analogous.}
\label{FIG:schematic}
\end{figure}


\section{Analysis of the hierarchies of upper bounds} \label{SEC:upbounds}
One method of analyzing the behaviour of the upper bounds \eqref{EQ:upbound}, \eqref{EQ:uppbound} and \eqref{EQ:upppbound} (employed, e.g., in \cite{deKlerkHessLaurent:hypercube, deKlerkLaurent:hypercube2010, deKlerkLaurent:annealing, deKlerkLaurentSun:upbounds, SlotLaurent:upbound}) is to exhibit an explicit probability density $\sigma \in \R[\x]$ on $(\mainset, \mu)$ which lies in the appropriate cone, and for which the difference
\[
	\int_{\mainset} f(\x) \sigma(\x) d\mu(\x) - \funcmin
\]
can be bounded from above. We exhibit such a $\sigma$ here for $\mainset=\ball{n}$ based on the perturbed Christoffel-Darboux kernels \eqref{EQ:overvieweq1} constructed in Section \ref{SEC:CD}.

Write $\kernel(\x, \y) = \kernelCDp_{2r}(\x, \y; \lambda)$ for such a kernel, where $\lambda$ is chosen as in Lemma~\ref{LEM:qanalysis}.
Let $\x^* \in \mainset$ be a global minimizer of $f$ over $\mainset$, and consider the polynomial $\sigma$ given by:
\[
	\sigma(\x) = \kernel(\x, \x^*) \quad (\x \in \mainset).
\]
By Lemma~\ref{LEM:ballsosrepresentation},  $\sigma \in \preordering{\mainset}_{2r}$. For any polynomial $p \in \R[\x]_d$, we have:
\[
	\int_{\mainset} \sigma(\x) p(\x) d\mu(\x) = \kernelop p(\x^*) = \sum_{k=0}^d \lambda_k p_k(\x^*) \quad (p_k \in H_k),
\]
by definition. Therefore, as $\kernelop$ satisfies \eqref{PROPERTY:normalization}, $\sigma$ is a probability density on $\mainset$. The polynomial $\sigma$ is thus a feasible solution to \eqref{EQ:upppbound}. Furthermore, we have:
\begin{align*}
	\int_{\mainset} f(\x) \sigma(\x) d\mu(\x) - \funcmin &= \kernelop f(\x^*) -  f(\x^*) 
	\leq \sum_{k=1}^d |(1- \lambda_k) f_k(\x^*)| \\
	&\leq \harmbound{\mainset}_d \cdot \sum_{k=1}^d |1- \lambda_k| \leq \harmbound{\mainset}_d \cdot \frac{n^2d^3}{r^2}.
\end{align*}
We find that $\upppbound{r} - \funcmin \leq \harmbound{\mainset}_d \cdot n^2d^3/r^2$. Theorem~\ref{THM:secondball} follows immediately. For $\mainset=\simplex{n}$, an analogous construction yields Theorem~\ref{THM:secondsimplex}.

\section{Concluding remarks}
\subsection*{Summary}
We have shown a convergence rate in $O(1/r^2)$ for the Schm\"udgen-type hierarchy of lower bounds $\lowwbound{r}$ for the minimization of a polynomial $f$ over the unit ball or simplex. Alternatively, if $\eta > 0$, we have shown a bound in $O(1/\sqrt{\eta})$ on the degree of a Schm\"udgen-type certificate of positivity for $f - \funcmin + \eta$.
Our result matches the recently obtained rates for the hypersphere \cite{FangFawzi:sphere} and the hypercube~ \cite{LaurentSlot:hypercube}. As a side result, we show similar convergence rates for the upper bounds \eqref{EQ:upppbound} on these sets as well (w.r.t. the measures $\mu_{\ball{}}$ and $\mu_{\simplex{}}$). We repeat that convergence rates in $O(1/r^2)$ for the upper bounds on $\ball{n}$ and $\simplex{n}$ were already available (but w.r.t. different reference measures).

\subsection*{Connection to the Christoffel-Darboux kernel}
Our proof technique establishes an explicit link between the Christoffel-Darboux kernel $\kernelCD_{2r}$ and the error of Lasserre's hierarchies at order $r$, which was implicitely present in the earlier works \cite{FangFawzi:sphere, SlotLaurent:bincube}. An analysis of this error is then possible whenever conditions can be derived ensuring the perturbed kernel $\kernelCDp_{2r}(\x, \y; \lambda)$ of \eqref{EQ:overvieweq1} lies in the truncated preordering $\preordering{\mainset}_{2r}$. In the present work and in \cite{FangFawzi:sphere, SlotLaurent:bincube}, such conditions are derived using a closed form expression of $\kernelCD_{2r}$. They are met when $\lambda$ is chosen according to the expansion of a \emph{univariate} sum of squares in the right basis of orthogonal polynomials, thus reducing the analysis to a univariate setting.

\subsection*{Comparison to the analysis on $[-1, 1]^n$}
For the analysis of the Schm\"udgen-type lower bounds on $[-1,1]^n$ in \cite{LaurentSlot:hypercube}, the authors make use of a multivariate Jackson kernel $\kernel_{2r}^{\rm jack}$ (see also \cite{PKMsurvey}). This kernel is defined in terms of multivariate Chebyshev polynomials $\{\mathcal{C}_\alpha : \alpha \in \N^n\}$ as:
\[
	\kernel^{\rm jack}_{2r}(\x, \y) = \sum_{|\alpha| \leq 2r} \lambda_{\alpha} \mathcal{C}_\alpha(\x) \mathcal{C}_\alpha(\y) \quad (\x, \y \in [-1,1]^n),
\]
for certain $\lambda_{\alpha} \in \R$. As the Chebyshev polynomials form an orthogonal basis for $\R[\x]$ w.r.t. the Chebysev measure on $[-1, 1]^n$, the Jackson kernel may again be viewed as a perturbed version of the Christoffel-Darboux kernel. Membership of $\kernel_{2r}^{\rm jack}(\x, \y)$ in the preordering is shown making use of the known nonnegativity of the (univariate) Jackson kernel, combined with a representation theorem for nonnegative polynomials on the interval $[-1, 1]$.

\subsection*{\LS{Putinar- vs. Schm\"udgen-type certificates}}
\LS{In light of the recent result~\cite{BaldiMourrain:putinar}, there is no longer a (large) theoretical gap between the best known convergence rates for the Putinar- and Schm\"udgen-type hierarchies.}
On the other hand, specialized convergence result for the Putinar-type bound $\lowbound{r}$ of \eqref{EQ:lowbound} are so far available only in those cases where $\quadmodule{\mainset}_{2r} = \preordering{\mainset}_{2r}$, i.e. where the Putinar- and Schm\"udgen-type bounds coincide (which is the case on the binary hypercube, the hypersphere and the unit ball). It is an interesting open question whether specialized results for $\lowbound{r}$ can be shown in non-trivial cases as well, for instance on the simplex and hypercube.
A major motivation for this question is the fact that the Putinar-type hierarchy is of relatively greater practical relevance. Indeed, computing the (stronger) Schm\"udgen-type bounds is often intractable, even for small values of~$r$. \LS{It seems unclear whether the techniques of the present paper may be applied to the Putinar-type bounds.  For instance, on the simplex, this would require an analog of Lemma~\ref{LEM:simplexsosrepresentation} showing membership of the kernel in the \emph{quadratic module} rather than the preordering. Such an analog seems difficult to prove using the representation \eqref{EQ:simplexqkernel}, and no obvious other representation is available.}

\subsection*{The harmonic constant}
The constant $\harmbound{\mainset}_d$ (which bounds the ratio  $\infnorm{\mainset}{p_k} / \infnorm{\mainset}{p}$ between the supremum norm of a polynomial $p \in \poly{\mainset}_d$ and that of its components $p_k \in H_k$) plays an important role in our analysis of Lasserre's hierarchies on $\ball{n}$ and $\simplex{n}$. We have shown in Section~\ref{SEC:harmbound} that $\harmbound{\ball{n}}_d$ and $\harmbound{\simplex{n}}_d$ depend polynomially on $n$ (for fixed $d$) and on $d$ (for fixed $n$).
The constants $\harmbound{S^{n-1}}_d$ and  $\harmbound{\{0, 1\}^n}_d$ for the sphere and binary cube may actually be bounded \emph{independently} of the dimension $n$ for fixed $d \in \N$ \cite{FangFawzi:sphere, SlotLaurent:bincube}. It is an interesting question whether this is true for $\harmbound{\ball{n}}_d$ and $\harmbound{\simplex{n}}_d$ as well.
In a subsequent paper, we will study the constant $\harmbound{\mainset}_d$ in more detail, focusing in particular on its asymptotic properties.

\subsection*{Acknowledgments}
We thank Monique Laurent and Fernando Oliviera for fruitful discussions. We further thank Monique Laurent for helpful comments on an earlier version of the manuscript. We thank the anonymous referees for their valuable suggestions. We thank Alison Surey and Markus Schweighofer for pointing out a technical error in the proof of Lemma~\ref{LEM:qanalysis}.

\appendix
\section{Sum-of-squares representations} \label{APP:sosrepresentations}
Here, we give a proof of Lemma~\ref{LEM:ballsosrepresentation} and Lemma~\ref{LEM:simplexsosrepresentation}.
Both proofs rely on the following lemma.
\begin{lemma}\label{LEM:sosrepresentationhelp}
Let $p \in \R[x]$ be a univariate polynomial of degree $r$. Let $\fvar, \fvarr$ be formal variables. Then the polynomial
$p(\fvar + \fvarr)^2 + p(\fvar - \fvarr)^2$ admits a representation:
\[
	p(\fvar + \fvarr)^2 + p(\fvar - \fvarr)^2 = \fvarr^2 h_{\mathrm{odd}}(\fvar, \fvarr^2)^2 + h_{\mathrm{even}}(\fvar, \fvarr^2)^2,
\]
where $\fvarr h_{\rm odd}(\fvar, \fvarr^2)$ and  $h_{\rm even}(\fvar, \fvarr^2) \in \R[\fvar, \fvarr]$ are polynomials of degree $r$.
\end{lemma}
\begin{proof}
For convenience, let $u \in \R[\fvar, \fvarr]$ be given by $u(\fvar, \fvarr) = \fvar + \fvarr$, so that:
\[
	p(\fvar + \fvarr)^2 + p(\fvar - \fvarr)^2 = p(u(\fvar, \fvarr))^2 +  p(u(\fvar, -\fvarr))^2 = h(\fvar, \fvarr)^2 + h(\fvar, -\fvarr)^2,
\]
where $h = p \circ u \in \R[\fvar, \fvarr]_r$. If we expand $h$ in the monomial basis of $\R[\fvar, \fvarr]$ as:
\[
	h(\fvar, \fvarr) = \sum_{i + j \leq r} h_{ij} \fvar^i \fvarr^j \quad (h_{ij} \in \R),
\]
we may perform the following computation (where all summations are taken over $i + j \leq r$):
\begin{align*}
h(\fvar, \fvarr)^2 &+ h(\fvar, -\fvarr)^2 \\
&= \big(\sum h_{ij} \fvar^i \fvarr^j\big)^2 + \big(\sum h_{ij} \fvar^i (-\fvarr)^j\big)^2 \\
&= \big(\sum_{j \text{ odd}} h_{ij} \fvar^i \fvarr^j + \sum_{j \text{ even}} h_{ij} \fvar^i \fvarr^j \big)^2 + \big(\sum_{j \text{ odd}} h_{ij} \fvar^i (-\fvarr)^j + \sum_{j \text{ even}} h_{ij} \fvar^i (-\fvarr)^j \big)^2 \\
&= \big(\sum_{j \text{ odd}} h_{ij} \fvar^i \fvarr^j + \sum_{j \text{ even}} h_{ij} \fvar^i \fvarr^j \big)^2 + \big(-\sum_{j \text{ odd}} h_{ij} \fvar^i \fvarr^j + \sum_{j \text{ even}} h_{ij} \fvar^i \fvarr^j \big)^2 \\
&=(\sum_{j \text{ odd}} h_{ij} \fvar^i \fvarr^j)^2 + (\sum_{j \text{ even}} h_{ij} \fvar^i \fvarr^j)^2 + 2 (\sum_{j \text{ odd}} h_{ij} \fvar^i \fvarr^j) (\sum_{j \text{ even}} h_{ij} \fvar^i \fvarr^j) \\
&\quad + (\sum_{j \text{ odd}} h_{ij} \fvar^i \fvarr^j)^2 + (\sum_{j \text{ even}} h_{ij} \fvar^i \fvarr^j)^2 - 2 (\sum_{j \text{ odd}} h_{ij} \fvar^i \fvarr^j) (\sum_{j \text{ even}} h_{ij} \fvar^i \fvarr^j) \\
&=2(\sum_{j \text{ odd}} h_{ij} \fvar^i \fvarr^j)^2 + 2(\sum_{j \text{ even}} h_{ij} \fvar^i \fvarr^j)^2 \\
&=2\fvarr^2(\sum_{j \text{ odd}} h_{ij} \fvar^i \fvarr^{j-1})^2 + 2(\sum_{j \text{ even}} h_{ij} \fvar^i \fvarr^j)^2. \\
\end{align*}
But now we see that there exist $h_{\mathrm{odd}}, h_{\mathrm{even}} \in \R[\fvar, \fvarr]$ of appropriate degree such that:
\[
p(\fvar + \fvarr)^2 + p(\fvar - \fvarr)^2 = h(\fvar, \fvarr)^2 + h(\fvar, -\fvarr)^2 = \fvarr^2 h_{\mathrm{odd}}(\fvar, \fvarr^2)^2 + h_{\mathrm{even}}(\fvar, \fvarr^2)^2,
\]
as required.
\end{proof}

\begin{lemma}[Restatement of Lemma~\ref{LEM:ballsosrepresentation}]
Let $q \in \Sigma[x]_{2r}$ be a univariate sum of squares. Then the kernel $\kernelCDp_{2r}(\x, \y; \lambda)$ in \eqref{EQ:ballqkernel} satisfies $\kernelCDp_{2r}(\cdot, \y; \lambda) \in \preordering{\ball{n}}_{2r}$ for fixed $\y \in \ball{n}$. 
\end{lemma}
\begin{proof}
We may assume w.l.o.g. that $q = p^2$ is a square.  
For $\x, \y \in \ball{n}$, write $\fvar = \x \cdot \y$ and $\fvarr = \sqrt{1 - \|\x\|^2} \sqrt{1- \|\y\|^2}$, so that:
\[
	\kernelCDp_{2r}(\x, \y; \lambda) = p(\fvar + \fvarr)^2 + p(\fvar - \fvarr)^2.
\]
By Lemma~\ref{LEM:sosrepresentationhelp}, there exist  $h_{\mathrm{odd}},  h_{\mathrm{even}} \in \R[\fvarr, \fvar]$ of appropriate degree so that:
\[
	\kernelCDp_{2r}(\x, \y; \lambda) = \fvarr^2 h_{\mathrm{odd}}(\fvar, \fvarr^2)^2 + h_{\mathrm{even}}(\fvar, \fvarr^2)^2,
\]
which lies in $\preordering{\ball{n}}_{2r}$ for fixed $\y \in B^n$ as $\fvar = \x \cdot \y$ and $\fvarr^2 = (1-\|\x\|^2) (1-\|\y\|^2)$.
\end{proof}

\begin{lemma}[Restatement of Lemma~\ref{LEM:simplexsosrepresentation}]
Let $q \in \Sigma[x]_{4r}$ be a univariate sum of squares of degree $4r$, and let $q_{\rm even}$ be as in \eqref{EQ:evenq}. Then the kernel $\kernelCDp_{2r}(\x, \y; \lambda_{\rm even})$ in \eqref{EQ:simplexqkernel} satisfies $\kernelCDp_{2r}(\cdot, \y; \lambda_{\rm even}) \in \preordering{\simplex{n}}_{2r}$ for fixed $\y \in \simplex{n}$.
\end{lemma}
\begin{proof}
Note first that $q_{\rm even}(x) = \frac{1}{2}\big(q(x) + q(-x)\big)$ is itself a sum of squares. In view of \eqref{EQ:simplexqkernel}, it now suffices to show that for any square $p^2$ of degree $4r$, the kernel:
\[
	\kernel(\x, \y) := \sum_{t \in \{-1, 1\}^{n+1}} p^2\big(\sum_{i=1}^{n+1} \sqrt{x_iy_i}t_i\big)
\]
lies in $\preordering{\simplex{n}}_{2r}$ for fixed $\y \in \simplex{n}$ (recall that we write $x_{n+1} = 1 - |\x|, y_{n+1} = 1-|\y|$). We prove this by successive application of Lemma~\ref{LEM:sosrepresentationhelp}. Set first $\fvar_{(0)} = \sum_{i=1}^{n} \sqrt{x_iy_i}t_i$ and $\fvarr_{(0)} = \sqrt{x_{n+1}y_{n+1}}$. Then Lemma~\ref{LEM:sosrepresentationhelp} tell us that:
\begin{align*}
	\kernel(\x, \y) &= \sum_{t \in \{-1, 1\}^{n+1}} p^2(\fvar_{(0)} + \fvarr_{(0)} t_{n+1}) \\
	&= \sum_{t \in \{-1, 1\}^n} \big(p^2(\fvar_{(0)} + \fvarr_{(0)}) + p^2(\fvar_{(0)} - \fvarr_{(0)}) \big) \\
	&= \sum_{t \in \{-1, 1\}^n} \big( \fvarr_{(0)}^2 h_{0}(\fvar_{(0)}, \fvarr_{(0)}^2)^2 + h_{1}(\fvar_{(0)}, \fvarr_{(0)}^2)^2 \big),
\end{align*}
for polynomials $h_0, h_1$ of appropriate degree.
Now, we may set $\fvar_{(1)} = \sum_{i=1}^{n-1} \sqrt{x_iy_i}t_i$ and $\fvarr_{(1)} = \sqrt{x_{n}y_{n}}$ and proceed to find:
\begin{alignat*}{2}
	\kernel(\x, \y) &= \sum_{t \in \{-1, 1\}^n} &&\big( \fvarr_{(0)}^2 h_{0}(\fvar_{(0)}, \fvarr_{(0)}^2)^2 + h_{1}(\fvar_{(0)}, \fvarr_{(0)}^2)^2 \big) \\
	&= 
\sum_{t \in \{-1, 1\}^{n-1}} &&\bigg(\fvarr_{(0)}^2 \big( h^2_{0}(\fvar_{(1)} + \fvarr_{(1)}, \fvarr_{(0)}^2) + h^2_{0}(\fvar_{(1)} - \fvarr_{(1)}, \fvarr_{(0)}^2 \big) \\ 
& &&+ h^2_{1}(\fvar_{(1)} + \fvarr_{(1)}, \fvarr_{(0)}^2) + h^2_{1}(\fvar_{(1)} - \fvarr_{(1)}, \fvarr_{(0)}^2) \bigg) \\
	&= 
\sum_{t \in \{-1, 1\}^{n-1}} &&\bigg(\fvarr_{(0)}^2\fvarr_{(1)}^2 h^2_{00}(\fvar_{(1)}, \fvarr_{(1)}^2, \fvarr_{(0)}^2) + \fvarr_{(0)}^2h^2_{01}(\fvar_{(1)}, \fvarr_{(1)}^2, \fvarr_{(0)}^2 \big)\\ 
& &&+ \fvarr_{(1)}^2 h^2_{10}(\fvar_{(1)}, \fvarr_{(1)}^2, \fvarr_{(0)}^2) + h^2_{11}(\fvar_{(1)}, \fvarr_{(1)}^2, \fvarr_{(0)}^2) \bigg).
\end{alignat*}
for polynomials $h_{00}, h_{01}, h_{10}, h_{11}$ of appropriate degree. After $n$ applications of this procedure, we find that:
\begin{equation} \label{EQ:APPsimplex}
	\kernel(\x, \y) = \sum_{a \in \{0, 1\}^{n+1}} \prod_{i=0}^{n} \fvarr_{(i)}^{2a_i} \cdot h^2_a(\fvarr_{(0)}^2, \ldots, \fvarr_{(n)}^2),
\end{equation}
for polynomials $h_a$ of appropriate degree, where $\fvarr_{(i)}^2 = x_{n+1-i}y_{n+1-i}$ for $0 \leq i \leq n$. As $\simplex{n} = \{\x \in \R^n : x_i \geq 0 \quad (1 \leq i \leq n+1) \}$, this means that $\kernel(\x, \y) \in \preordering{\simplex{n}}_{2r}$ for $\y \in \simplex{n}$ fixed. Take note that while the summands in \eqref{EQ:APPsimplex} are of degree $4r$ in the variables $\fvarr_{(1)}, \ldots, \fvarr_{(n)}$, they indeed have degree $2r$ in the variables ${x_1, \ldots, x_n}$.
\end{proof}

\bibliographystyle{amsplain}
\bibliography{bibLucas}

\end{document}